\documentclass{amsart}

\usepackage{amsfonts}
\usepackage{mathrsfs}
\usepackage{amsmath}
\usepackage{amsthm}
\usepackage{amssymb,bbm}
\usepackage{latexsym,hyperref}
\usepackage{enumitem}
\usepackage[all]{xy}

\usepackage{etoolbox}
\usepackage[mathscr]{euscript}

\hypersetup{%
    pdfborder = {0 0 0},
    colorlinks,
    citecolor=red!50!black,
    filecolor=Darkgreen,
    linkcolor=blue!40!black,
    urlcolor=cyan!50!black!90
}

\numberwithin{equation}{section}

\newtheorem{thrm}{Theorem}[section]
\newtheorem{prop}[thrm]{Proposition}
\newtheorem{crl}[thrm]{Corollary}

\theoremstyle{definition}
\newtheorem{defn}[thrm]{Definition}

\theoremstyle{remark}

\newcommand{\nc}{\newcommand}

\nc{\al}{\alpha}
\nc{\del}{\delta}
\nc{\eps}{\epsilon}
\nc{\ga}{\gamma}
\nc{\Ga}{\Gamma}
\nc{\ka}{\kappa}
\nc{\la}{\lambda}
\nc{\om}{\omega}
\nc{\si}{\sigma}
\nc{\Si}{\Sigma}
\nc{\bsi}{\boldsymbol\sigma}
\nc{\bSi}{\boldsymbol\Sigma}
\nc{\Ups}{\upsilon}
\nc{\vphi}{\varphi}

\nc{\btau}{\boldsymbol\tau}
\nc{\bdel}{\boldsymbol\delta}

\nc{\id}{\mathrm{id}}
\nc{\gr}{\mathrm{gr}}

\nc{\Ug}{U\mathfrak{g}}
\nc{\Ub}{U\mathfrak{b}}

\nc{\Hk}{\mathsf{H}}
\nc{\ombH}{\overline{\mathbf{H}}}

\nc{\ud}{\underline}
\nc{\tl}{\tilde}

\nc{\mbA}{\mathbf{A}}
\nc{\mbb}{\mathbf{b}}
\nc{\mbB}{\mathbf{B}}
\nc{\mbc}{\mathbf{c}}
\nc{\mbC}{\mathbf{C}}
\nc{\mbd}{\mathbf{d}}
\nc{\mbD}{\mathbf{D}}
\nc{\mbe}{\mathbf{e}}
\nc{\mbE}{\mathbf{E}}
\nc{\mbf}{\mathbf{f}}
\nc{\mbF}{\mathbf{F}}
\nc{\mbg}{\mathbf{g}}
\nc{\mbH}{\mathbf{H}}
\nc{\mbh}{\mathbf{h}}
\nc{\mbi}{\mathbf{i}}
\nc{\mbI}{\mathbf{I}}
\nc{\mbj}{\mathbf{j}}
\nc{\mbJ}{\mathbf{J}}
\nc{\mbk}{\mathbf{k}}
\nc{\mbK}{\mathbf{K}}
\nc{\mbL}{\mathbf{L}}
\nc{\mbM}{\mathbf{M}}
\nc{\mbQ}{\mathbf{Q}}
\nc{\mbq}{\mathbf{q}}
\nc{\mbr}{\mathbf{r}}
\nc{\mbT}{\mathbf{T}}
\nc{\mbu}{\mathbf{u}}
\nc{\mbU}{\mathbf{U}}
\nc{\mbv}{\mathbf{v}}
\nc{\mbV}{\mathbf{V}}
\nc{\mbw}{\mathbf{w}}
\nc{\mbW}{\mathbf{W}}
\nc{\mbX}{\mathbf{X}}
\nc{\mbY}{\mathbf{Y}}
\nc{\mbZ}{\mathbf{Z}}

\nc{\mbbA}{\mathbb{A}}
\nc{\mbbB}{\mathbb{B}}
\nc{\mbbD}{\mathbb{D}}
\nc{\mbbF}{\mathbb{F}}
\nc{\mbbV}{\mathbb{V}}
\nc{\mbbH}{\mathbb{H}}
\nc{\mbbK}{\mathbb{K}}
\nc{\mbbL}{\mathbb{L}}
\nc{\mbbP}{\mathbb{P}}
\nc{\mbbU}{\mathbb{U}}

\nc{\mcA}{\mathcal{A}}
\nc{\mcB}{\mathcal{B}}
\nc{\mcC}{\mathcal{C}}
\nc{\mcD}{\mathcal{D}}
\nc{\mcE}{\mathcal{E}}
\nc{\mcF}{\mathcal{F}}
\nc{\mcG}{\mathcal{G}}
\nc{\mcH}{\mathcal{H}}
\nc{\mcK}{\mathcal{K}}
\nc{\mcN}{\mathcal{N}}
\nc{\mcO}{\mathcal{O}}
\nc{\mcQ}{\mathcal{Q}}
\nc{\mcS}{\mathcal{S}}
\nc{\mcP}{\mathcal{P}}
\nc{\mcU}{\mathcal{U}}
\nc{\mcT}{\mathcal{T}}
\nc{\mcV}{\mathcal{V}}
\nc{\mcW}{\mathcal{W}}
\nc{\mcX}{\mathcal{X}}
\nc{\mcY}{\mathcal{Y}}
\nc{\mcZ}{\mathcal{Z}}

\nc{\mfa}{\mathfrak{a}}
\nc{\mfA}{\mathfrak{A}}
\nc{\mfb}{\mathfrak{b}}
\nc{\mfB}{\mathfrak{B}}
\nc{\mfC}{\mathfrak{C}}
\nc{\mfd}{\mathfrak{d}}
\nc{\mfD}{\mathfrak{D}}
\nc{\mfe}{\mathfrak{e}}
\nc{\mfE}{\mathfrak{E}}
\nc{\mff}{\mathfrak{f}}
\nc{\mfF}{\mathfrak{F}}
\nc{\mfg}{\mathfrak{g}}
\nc{\mfgl}{\mathfrak{g}\mathfrak{l}}
\nc{\mfh}{\mathfrak{h}}
\nc{\mfH}{\mathfrak{H}}
\nc{\mfJ}{\mathfrak{J}}
\nc{\mfk}{\mathfrak{k}}
\nc{\mfK}{\mathfrak{K}}
\nc{\mfl}{\mathfrak{l}}
\nc{\mfL}{\mathfrak{L}}
\nc{\mfM}{\mathfrak{M}}
\nc{\mfm}{\mathfrak{m}}
\nc{\mfn}{\mathfrak{n}}
\nc{\mfN}{\mathfrak{N}}
\nc{\mfo}{\mathfrak{o}}
\nc{\mfP}{\mathfrak{P}}
\nc{\mfQ}{\mathfrak{Q}}
\nc{\mfS}{\mathfrak{S}}
\nc{\mfsl}{\mathfrak{s}\mathfrak{l}}
\nc{\mfso}{\mathfrak{s}\mathfrak{o}}
\nc{\mfsp}{\mathfrak{s}\mathfrak{p}}
\nc{\mft}{\mathfrak{t}}
\nc{\mfU}{\mathfrak{U}}
\nc{\mfu}{\mathfrak{u}}
\nc{\mfUqsl}{\mathfrak{U}_q\mathfrak{sl}}
\nc{\mfUsl}{\mathfrak{Usl}}
\nc{\mfV}{\mathfrak{V}}
\nc{\mfX}{\mathfrak{X}}
\nc{\mfY}{\mathfrak{Y}}
\nc{\mfz}{\mathfrak{z}}

\nc{\mrmd}{\mathrm{d}}

\nc{\sal}{\check{\alpha}}
\nc{\cbeta}{\check{\beta}}
\nc{\cd}{\check{d}}
\nc{\cf}{\check{f}}
\nc{\cdelta}{\check{\delta}}
\nc{\ccr}{\check{r}}
\nc{\cs}{\check{s}}

\nc{\bv}{\breve{v}}

\nc{\tc}{\tilde{c}}
\nc{\tr}{\tilde{r}}
\nc{\ts}{\tilde{s}}
\nc{\tv}{\tilde{v}}

\nc{\msA}{\mathsf{A}}
\nc{\msB}{\mathsf{B}}
\nc{\msC}{\mathsf{C}}
\nc{\msc}{\mathsf{c}}
\nc{\msD}{\mathsf{D}}
\nc{\msd}{\mathsf{d}}
\nc{\mse}{\mathsf{e}}
\nc{\msw}{\mathsf{w}}
\nc{\msq}{\mathsf{q}}
\nc{\msg}{\mathsf{g}}
\nc{\msE}{\mathsf{E}}
\nc{\msf}{\mathsf{f}}
\nc{\msF}{\mathsf{F}}
\nc{\msh}{\mathsf{h}}
\nc{\msk}{\mathsf{k}}
\nc{\msH}{\mathsf{H}}
\nc{\msI}{\mathsf{I}}
\nc{\msJ}{\mathsf{J}}
\nc{\msK}{\mathsf{K}}
\nc{\msL}{\mathsf{L}}
\nc{\msP}{\mathsf{P}}
\nc{\msQ}{\mathsf{Q}}
\nc{\msR}{\mathsf{R}}
\nc{\mss}{\mathsf{s}}
\nc{\msS}{\mathsf{S}}
\nc{\msT}{\mathsf{T}}
\nc{\msU}{\mathsf{U}}
\nc{\msv}{\mathsf{v}}
\nc{\msV}{\mathsf{V}}
\nc{\msx}{\mathsf{x}}
\nc{\msX}{\mathsf{X}}
\nc{\msY}{\mathsf{Y}}
\nc{\msZ}{\mathsf{Z}}

\nc{\End}{\mathrm{End}}
\nc{\Ext}{\mathrm{Ext}}
\nc{\Hom}{\mathrm{Hom}}
\nc{\Ima}{\mathrm{Image}}
\nc{\Ind}{\mathrm{Ind}}
\nc{\Ker}{\mathrm{Ker}}
\nc{\RHom}{\mathrm{RHom}}
\nc{\Sym}{\mathrm{Sym}}

\nc{\mtc}{\mathtt{c}}
\nc{\mtD}{\mathtt{D}}
\nc{\mte}{\mathtt{e}}
\nc{\mtE}{\mathtt{E}}
\nc{\mtf}{\mathtt{f}}
\nc{\mtF}{\mathtt{F}}
\nc{\mth}{\mathtt{h}}
\nc{\mtH}{\mathtt{H}}
\nc{\mtV}{\mathtt{V}}
\nc{\mtX}{\mathtt{X}}
\nc{\mty}{\mathtt{y}}

\nc{\ddeg}{\mathtt{deg}}
\nc{\dimm}{\mathtt{dim}}
\nc{\lmod}{\mathtt{lmod}}
\nc{\opp}{\mathtt{opp}}
\nc{\rmod}{\mathtt{rmod}}

\nc{\mmod}{\mathrm{mod}}
\nc{\nbh}{\mathrm{nbh}}

\nc{\lrh}{\leftrightharpoons}
\nc{\iso}{\stackrel{\sim}{\longrightarrow}}
\nc{\liso}{\stackrel{\sim}{\longleftarrow}}
\nc{\wh}{\widehat}
\nc{\wt}{\widetilde}
\nc{\lra}{\longrightarrow}
\nc{\ra}{\rightarrow}
\nc{\into}{\hookrightarrow}
\nc{\onto}{\twoheadrightarrow}

\nc{\C}{\mathbb{C}}
\nc{\N}{\mathbb{N}}
\nc{\Z}{\mathbb{Z}}

\nc{\SW}{\mathsf{SW}}

\nc{\ot}{\otimes}
\nc{\op}{\oplus}

\nc{\lan}{\langle}
\nc{\ran}{\rangle}

\nc{\qu}{\quad}
\nc{\qq}{\qquad}

\nc{\mysum}{\textstyle\sum}
\nc{\mysuml}{\textstyle\sum\limits}

\nc{\Xp}{x_{j_1,\ldots,j_p}^{-1}}
\nc{\Xeta}{x_{\eta_1,\ldots,\eta_e}^{-1}}

\nc{\spl}[1]{\begin{align}\begin{split}#1\end{split}\end{align}}
\nc{\eqa}[1]{\begin{align}#1\end{align}}
\nc{\eqn}[1]{\begin{align*}#1\end{align*}}
\nc{\eq}[1]{\begin{equation}#1\end{equation}}

\usepackage{color}
\usepackage[usenames,dvipsnames]{xcolor}
\nc{\red}{\color{red}}
\nc{\blu}{\color{blue}}
\nc{\br}{\color{Brown}}
\nc{\gre}{\color{green!50!black}}
\nc\el{\nonumber\\}
\nc\nn{\nonumber}

\nc\bb{\mathbb}

\nc\mf[1]{\mathfrak{#1}}

\nc\Tr{{\rm tr}}

\nc{\sm}[1]{\text{\tiny{\rm #1}}}

\nc\tdeg{\stackrel{\sim}{\smash{\deg}\rule{0pt}{1.1ex}}}

\renewcommand{\,}{\kern 0.12em} 

\begin{document}

\hfill DMUS-MP-16/03

\vspace{1.2cm}

\begin{center}
{\Large{\textbf{Twisted Yangians of small rank}}} 

\bigskip

Nicolas Guay$^{1a}$, Vidas Regelskis$^{2b}$, Curtis Wendlandt$^{1c}$

\end{center}

\bigskip

\begin{center} \small 
$^1$ University of Alberta, Department of Mathematics, CAB 632, Edmonton, AB T6G 2G1, Canada.\\ $^2$ University of Surrey, Department of Mathematics, Guildford, GU2 7XH,  UK. \\
\smallskip
E-mail: $^a$\,nguay@ualberta.ca $^b$\,v.regelskis@surrey.ac.uk $^c$\,cwendlan@ualberta.ca
\end{center}

\patchcmd{\abstract}{\normalsize}{}{}{}

\begin{abstract} \small 
We study quantized enveloping algebras called twisted Yangians associated with the symmetric pairs of types CI, BDI and DIII (in Cartan's classification) when the rank is small. We establish isomorphisms between these twisted Yangians and the well known Olshanskii's twisted Yangians of types AI and AII, and also with the Molev-Ragoucy reflection algebras associated with symmetric pairs of type AIII. We also construct isomorphisms with twisted Yangians in Drinfeld's original presentation. 
\end{abstract}

\makeatletter
\@setabstract
\makeatother

\medskip

\thispagestyle{empty}


\section{Introduction}

The twisted Yangians $Y^\pm(N)$ introduced by G. Olshanskii in \cite{Ol} are quantizations of the enveloping algebra of the twisted current Lie algebra $\mfgl_N[u]^\rho$ where $\rho$ is an involution of $\mfgl_N$ (see \cite{MNO} and reference therein). The cases when $N$ is small are already interesting: for instance, $Y^\pm(2)$  controls the branching multiplicities of Gelfand-Zsetlin bases for the orthogonal and symplectic Lie algebras \cite{Mo3} and it plays an important role in understanding the representation theory of $Y^\pm(N)$ \cite{Mo2}, in particular in the classification of its finite dimensional, irreducible representations. 

In a recent work by the first two authors \cite{GR}, a class of extended twisted Yangians $X(\mfg,\mcG)^{tw}$ for symmetric pairs of types B, C and D was introduced. These new extended twisted Yangians are deformations of the enveloping algebra of the twisted current Lie algebras $\mfso_N[u]^\rho$ and $\mfsp_N[u]^\rho$ (where $\rho$ is now an involution of either $\mfso_N$ or $\mfsp_N$), and can be realised as coideal subalgebras of the orthogonal and symplectic extended Yangians $X(\mfo_N)$ and $X(\mfsp_N)$ studied by Arnaudon {\it et.\@ al.} in \cite{AACFR,AMR}. 

The representation theory of the extended Yangians was investigated in \cite{AMR}; it relies on the existence of isomorphisms between the extended Yangians $X(\mfsp_2)$, $X(\mfso_3)$ and $X(\mfso_4)$ and the Yangians $Y(2)$ and $Y(2)\ot Y(2)$. These isomorphisms are naturally expected to hold because of the isomorphisms $\mfsp_2 \cong \mfsl_2 \cong \mfso_3$ and $\mfso_4 \cong \mfsl_2 \oplus \mfsl_2$. Since finite dimensional modules over $Y(2)$ have been classified for a long time \cite{Ta1,Ta2}, those isomorphisms immediately lead to a similar classification for finite dimensional, irreducible modules over the extended Yangians $X(\mfsp_2)$, $X(\mfso_3)$ and $X(\mfso_4)$. This is then used in \cite{AMR} as the base case of an induction argument to obtain a similar classification for $X(\mfso_N)$ and $X(\mfsp_N)$. 

In the present work, we investigate the existence of similar isomorphisms for the extended twisted Yangians of rank one or two that correspond to isomorphisms between symmetric pairs of small rank (see Chapter X in \cite{He}). Since finite dimensional irreducible representations of $Y^\pm(2)$ were classified in \cite{Mo1}, we immediately obtain a similar classification for $X(\mfg,\mcG)^{tw}$ with $\mfg=\mfsp_2,\mfso_3,\mfso_4$. This will play an important role in the study of the representation theory of extended Yangians of arbitrary rank in \cite{GRW} and is the main motivation behind the current paper. Another motivation for the present work is to answer a natural question: are some of the twisted Yangians introduced in \cite{GR} isomorphic to some of those that appeared earlier in the literature? We show that the answer is positive for certain classical Lie algebras of rank one or two.

The paper is organized as follows. Section 2 introduces the notation and recalls the basic facts about symmetric pairs of classical types and the reflection equation, whereas Section 3 recalls the definition of Yangians and twisted Yangians and some properties of their algebraic structure. The main section is the fourth one where isomorphisms between various twisted Yangians associated to symmetric pairs of semisimple Lie algebras of rank one or two are obtained. In the last section, these isomorphisms are made explicit with twisted Yangians defined using the generators in Drinfeld's original presentation of Yangians (see Definition \ref{JdefY2}). This presentation is commonly used in integrable field theories and we have included that section because it may be useful to theoretical physicists working in that area.

\smallskip
 
{\it Acknowledgements.} The first author acknowledges the support of the Natural Sciences and Engineering Research Council of Canada through its Discovery Grant program.  Part of this work was done during the second author's visits to the University of Alberta. V.R. thanks the University of Alberta for the hospitality, and also gratefully acknowledges the Engineering and Physical Sciences Research Council (EPSRC) of the United Kingdom for the Postdoctoral Fellowship under the grant EP/K031805/1.

\bigskip


\section{Reflection equation and symmetric pairs}


\subsection{Preliminaries} \label{Sec:2.1s}

Let $n \in\N$ and set $N=2n$ or $N=2n+1$. We will always assume that $\mfg=\mfgl_N$, $\mfg=\mfsl_N$ or $\mfg=\mfg_N$, where $\mfg_N$ is the orthogonal Lie algebra $\mfso_N$ or the symplectic Lie algebra $\mfsp_N$ (only when $N=2n$). The algebra $\mfg_N$ can be realized as a Lie subalgebra of $\mfgl_N$ as follows. We label the rows and columns of $\mfgl_N$ by the indices $\{ -n,\ldots,-1,1,\ldots,n\}$ if $N=2n$ and by $\{-n,\ldots,-1,0,1,\ldots,n \}$ if $N=2n+1$. Set $\theta_{ij}=1$ in the orthogonal case $\forall \, i,j$ and $\theta_{ij}=\mathrm{sign}(i)\cdot \mathrm{sign}(j)$ in the symplectic case for $i,j\in\{ \pm 1, \pm 2, \ldots, \pm n  \}$. Let $F_{ij}=E_{ij} - \theta_{ij} E_{-j,-i}$ where $E_{ij}$ is the usual elementary matrix of $\mfgl_N$. Then $\mfg_N = \mathrm{span}_{\C} \{ F_{ij} \, | \, -n\le i,j\le n \}$. These matrices satisfy the relations 
\eq{ \label{[F,F]}
F_{ij} + \theta_{ij}F_{-j,-i}=0, \qquad [F_{ij},F_{kl}] = \delta_{jk}F_{il} - \delta_{il}F_{kj} + \theta_{ij}\delta_{j,-l}F_{k,-i} - \theta_{ij}\delta_{i,-k}F_{-j,l}.
}
Next, we need to introduce some operators: $P \in \End(\C^N \ot\C^N)$ will denote the permutation operator on $\C^N \ot \C^N$ and $Q=P^{t_1}=P^{t_2}$ where the transpose $t$ is given by $(E_{ij})^t = \theta_{ij} E_{-j,-i}$; explicitly,
\eq{ \label{PQ}
P=\mysum_{i,j=-n}^n E_{ij} \ot E_{ji}, \qquad Q = \mysum_{i,j=-n}^n \theta_{ij} E_{ij} \ot E_{-i,-j} .
}
We will denote the transpose when $\theta_{ij}=1$ by $t_+$ and the one when $\theta_{ij} = \mathrm{sign}(i)\cdot \mathrm{sign}(j)$ by $t_-$. Let $I$ denote the identity matrix. Then $P^2=I$ and $PQ=QP=\pm Q$ and $Q^2=N Q$, which will be useful below. Here (and further in this paper) the upper sign corresponds to the orthogonal case and the lower sign to the symplectic case.


\subsection{Symmetric pairs of classical Lie algebras and isomorphisms of small rank} \label{Sec:2.2}

The symmetric pairs we are interested in are of the form $(\mfg,\mfg^{\rho})$, where $\rho$ is an involution of $\mfg$ and $\mfg^\rho$ denotes the $\rho$-fixed subalgebra of $\mfg$ given by $\mfg^\rho = \{ X \in \mfg \,|\, X = \mcG X \mcG^{-1} \}$, except that $\mfg^\rho = \{ X \in \mfg \,|\, X = -X^t  \}$ for types AI and AII. Here, $\mcG$ is a certain matrix. All this data is listed below according to the type of the symmetric pair: 
\begin{itemize} [itemsep=.25ex]

\item AI{\phantom{IIII}} : $\mfg=\mfsl_N$ and $\mfg^\rho=\mfso_N$ with $t=t_+$; 

\item AII{\phantom{III}} : $\mfg=\mfsl_N$ and $\mfg^\rho=\mfsp_N$ with $t=t_-$;

\item AIII{\phantom{II}} : $\mfg=\mfsl_N$ and $\mfg^\rho=(\mfgl_p \op \mfgl_q)\cap \mfsl_N$ with $\mcG=\mysum_{i=1}^{p} E_{ii} - \mysum_{i=p+1}^N E_{ii}$; 

\item CI{\phantom{IIII}} : $N$ is even, $\mfg=\mfsp_N$, $\mcG=\mysum_{i=1}^{\frac{N}{2}} (E_{ii} - E_{-i,-i})$ and  $\mfg^{\rho} \cong \mfgl_{\frac{N}{2}}$;

\item DIII{\phantom{II}} : $N$ is even, $\mfg=\mfso_N$, $\mcG=\mysum_{i=1}^{\frac{N}{2}} (E_{ii} - E_{-i,-i})$ and  $\mfg^{\rho} \cong \mfgl_{\frac{N}{2}}$;

\item CII{\phantom{IIII}}:  $N$, $p$ and $q$ are even and $>0$, $N=p+q$, $\mfg=\mfsp_N$, 
\[
\textstyle \mcG= -\mysum_{i=1}^{\frac{q}{2}} (E_{ii} + E_{-i,-i}) + \mysum_{i=\frac{q}{2}+1}^{\frac{N}{2}} (E_{ii} + E_{-i,-i}) 
\]
and $\mfg^{\rho}= \mfsp_p \op \mfsp_q$. More precisely, the subalgebra of $\mfg^{\rho}$ spanned by $F_{ij}$ with $-\frac{q}{2} \le i,j \le \frac{q}{2}$ is isomorphic to $\mfsp_q$ and the subalgebra of $\mfg^{\rho}$ spanned by $F_{ij}$ with $|i|,|j| > \frac{q}{2}$ is isomorphic to $\mfsp_p$;

\item BDI{\phantom{II}} : $\mfg=\mfso_N$, $\mfg^{\rho}= \mfso_p \op \mfso_q$ where $p>q>0$ if $N$ is odd, and $p\geq q>0$ if $N$ is even (if $q=1$, then $\mfso_1$ is the zero Lie algebra), and $\mcG$ is
$$
\mcG=\begin{cases}
\mysum_{i=1}^{\frac{p-q}{2}} (E_{ii} + E_{-i,-i}) + \mysum_{i=\frac{p-q}{2}+1}^{\frac{N}{2} } (E_{-i,i} + E_{i,-i}) \quad \text{ for $N$ even,} \\[.5em]
\mysum_{i=-\frac{p-q-1}{2}}^{\frac{p-q-1}{2}} E_{ii}
 + \mysum_{i=\frac{p-q+1}{2}}^{\frac{N-1}{2}} (E_{-i,i} + E_{i,-i}) \quad\text{ for $N$ odd.}
 \end{cases}
$$
\end{itemize}
In addition, we will consider trivial symmetric pairs:
\begin{itemize}
\item ABCD0: $\mcG=I_N$, $\rho=\mathrm{id}$ and $\mfg^{\rho} = \mfg$.
\end{itemize}
These can be thought of as special cases of types AIII, CII and BDI when $p=N$ and $q=0$. For ease of notation, we will further refer to types CI, DIII, CII, BDI and BCD0 collectively as types B--C--D. We will always consider type A0 as a special case of type AIII.

\smallskip

\noindent When $n=1$ or $n=2$, we have the following well-known isomorphisms between classical Lie algebras, listed with their Dynkin type:
\begin{itemize} [itemsep=.5ex]

\item A($N=1$) = D($n=1$) \hspace{1.89cm} : $\mf{gl}_1 \cong \mf{so}_2$,

\item A($N=2$) = B($n=1$) = C($n=1$)\hspace{.14cm} : $\mf{sl}_2 \cong \mf{so}_3 \cong \mf{sp}_2$, 

\item D($n=2$) = A($N=2$) $\times$ A($N=2$) : $\mf{so}_4 \cong \mf{sl}_2 \op \mf{sl}_2$.

\end{itemize}
These lead to the following isomorphisms of symmetric pairs (see \cite[Section X.6]{He}):
\begin{itemize} 

\item AI($N=2$) = AIII($p=q=1$) = BDI ($p=2,q=1$) = CI($n=1$):
\eq{ \label{iso:1}
(\mf{sl}_2,\mf{so}_2 )\cong (\mf{sl}_2,\mf{gl}_1 ) \cong (\mf{so}_3,\mf{so}_2) \cong (\mf{sp}_2,\mf{gl}_1),
}

\item BDI($p=2, q=2$) = AI($N=2$) $\times$ AI($N=2$):
\eq{ \label{iso:2}
(\mf{so}_4,\mf{so}_2\op \mf{so}_2) \cong (\mf{sl}_2,\mf{so}_2) \oplus (\mf{sl}_2,\mf{so}_2),
}

\item DIII($n=2$) = AII($N=2$) $\times$ AI($N=2$):
\eq{ \label{iso:3}
(\mf{so}_4,\mf{gl}_2) \cong (\mf{sl}_2,\mf{sp}_2) \oplus (\mf{sl}_2,\mf{gl}_1),
}

\item BDI($p=3, q=1$) = A$^{\Delta}$($N=2$):
\eq{ \label{iso:4}
(\mf{so}_4,\mf{so}_3) \cong (\mf{sl}_2 \op \mf{sl}_2 , \Delta\mf{sl}_2),
}
where by A$^{\Delta}$ we mean the diagonal symmetric pair defined as follows: the Lie algebra $\mfsl_N\op\mfsl_N$ admits the involution $\rho$ given by $\rho(X_1,X_2) \mapsto (X_2,X_1)$ and $(\mfsl_N\op\mfsl_N)^{\rho} = \{ (X_1,X_2) \in \mfsl_N\op\mfsl_N \, | \, X_1 = X_2\}$; this fixed-point subalgebra is denoted $\Delta \mfsl_N$.

\end{itemize}

The isomorphisms $\mfso_5 \cong \mfsp_4$ and $\mfso_6 \cong \mfsl_4$ also lead to isomorphisms between certain symmetric pairs. It is natural to conjecture that there are also isomorphisms between corresponding twisted Yangians, but we do not explore this question here.


\subsection{R-matrices, K-matrices and the reflection equation} \label{Sec:2.3}

The matrices $R(u)\in\End(\C^N\ot \C^N)[[u^{-1}]]$ (all endomorphisms are over $\C$, unless specified otherwise) that we will need are defined by \cite{MNO,AACFR}:
\eq{
a) \qu R(u) = I - \frac{P}{u} \qu\text{for}\qu \mfgl_N, \qq b)\qu R(u) = I - \frac{P}{u} + \frac{Q}{u-\ka} \qu\text{for}\qu \mfg_N,  \label{R(u)}
}
where $\ka=N/2\mp 1$ for $\mfg=\mfg_N$; we also set $\ka=0$ for $\mfg=\mfgl_N$. These $R$-matrices are solutions of the quantum Yang-Baxter equation with spectral parameters,
\eq{ \label{YBE}
R_{12}(u)\, R_{13}(u+v)\, R_{23}(v) = R_{23}(v)\, R_{13}(u+v)\, R_{12}(u).
} 

We introduce $K$-matrices $K(u)\in \End(\C^N)[[u^{-1}]]$ by (see \cite{GR,MNO,MR})
\begin{itemize}
\item  $K(u)=I$ for symmetric pairs of types AI, AII;
\item $K(u)=\mcG$ for symmetric pairs of types AIII, BCD0, CI, DIII and DI, CII when $p=q$;
\item $K(u) = (I-c\,u\, \mcG)(1-c\,u)^{-1}$ with $c=\tfrac{4}{p-q}$ for symmetric pairs of types BDI, CII when $p>q$. 
\end{itemize}
The matrices $K(u)$ of types AI and AII are (scalar) solutions of the twisted reflection equation 
\eq{ \label{REt}
R(u-v)\,K_1(u)\,R(-u-v)^{t}\,K_2(v) = K_2(v)\,R(-u-v)^{t}\,K_1(u)\,R(u-v) ,
}
where $R(u)$ is given by a) in \eqref{R(u)}. In all other cases, $K(u)$ is a solution of the reflection equation 
\eq{ \label{RE}
R(u-v)\,K_1(u)\,R(u+v)\,K_2(v) = K_2(v)\,R(u+v)\,K_1(u)\,R(u-v) ,
}
with $R(u)$ given by a) in \eqref{R(u)} for type AIII, or by b) in \eqref{R(u)} for all the B--C--D types.


\section{Yangians and twisted Yangians}

In this section, we briefly summarize the results presented in \cite{MNO,MR} for Yangians and twisted Yangians of type A, and in \cite{AACFR,AMR,GR} for types B, C and D.


\subsection{Yangians and Extended Yangians} \label{Sec:3.1}

\begin{defn} \label{D:X(g)}
The algebra $X(\mfg)$ is the associative $\C$-algebra with generators $t_{ij}^{(r)}$ for $-n\le i,j\le n$ and $r\in\Z_{\ge 0}$, which satisfy the following relations: 
\eq{ \label{RTT}
R(u-v)\,T_1(u)\,T_2(v) = T_2(v)\,T_1(u)\,R(u-v), 
}
where $T(u)$ are the elements of $\End(\C^N) \ot X(\mfg)[[u^{-1}]] $ with matrix entries given by the formal power series 
\eq{
t_{ij}(u) = \mysum_{r=0}^{\infty} t_{ij}^{(r)}\, u^{-r} \in X(\mfg)[[u^{-1}]], \qq t_{ij}^{(0)} = \delta_{ij}.
}
\end{defn}
\noindent If $R(u)$ is given by a) of \eqref{R(u)}, then the algebra $X(\mfg)$ is the Yangian $Y(N)$ of $\mfgl_N$ \cite{Dr1,MNO}. Otherwise, if $R(u)$ is given by b) of \eqref{R(u)}, then the algebra $X(\mfg)$ is the extended Yangian $X(\mfg_N)$ \cite{AACFR,AMR}. \smallskip

Let us collect some important facts about $X(\mfg)$ and its centre: 

\begin{itemize}

\item Let $\mfg=\mfgl_N$. Then $X(\mfgl_N)=Y(N)$ and there exists a series, called quantum determinant
\eq{
{\rm qdet}\, T(u) = 1 + \mysum_{i\ge1} d_i u^{-i}\in Y(N)[[u^{-1}]] \label{qdetglN}
}
with coefficients central in $Y(N)$. Moreover, they are algebraically independent and generate the whole centre $Z(N)$ of $Y(N)$. The quotient algebra $Y(N)/({\rm qdet}\, T(u)-1)$ is isomorphic to the special Yangian $SY(N)$ associated with the Lie algebra $\mfsl_N$. 

\smallskip

\item Let $\mfg=\mfg_N$. Then the $T$-matrix $T(u)$ of $X(\mfg_N)$ satisfies 
\eq{
T(u+\ka)^{t}\,T(u) = T(u)\,T(u+\ka)^t = z(u)\cdot I, \label{TTtz(u)}
}
where the series $z(u)=1+\mysum_{i\geq1} z_i\,u^{-i}\in X(\mfg_N)[[u^{-1}]]$ is called the quantum contraction of $T(u)$. Its coefficients $z_i$ are algebraically independent and generate the whole centre $Z(\mfg_N)$ of $X(\mfg_N)$.  The quotient algebra $X(\mfg_N)/(z(u)-1)$ is isomorphic to the usual Yangian $Y(\mfg_N)$ of $\mfg_N$.
\end{itemize}

By employing the uniform notation
\eq{
(X(\mfg), \;Z(\mfg), \; Y(\mfg)) = \begin{cases} 
(Y(N), \hspace{.26cm} Z(N), \hspace{.21cm}  SY(N) ) & \text{for } \mfg=\mfgl_N, \\
(X(\mfg_N), \; Z(\mfg_N), \; Y(\mfg_N) ) & \text{for } \mfg=\mfg_N ,
\end{cases}
}
the following isomorphism of algebras holds
\eq{
X(\mfg) \cong Z(\mfg) \ot Y(\mfg) . \label{X=Z*Y}
}
We will denote the $T$-matrix of $Y(\mfg)$ by $\mcT(u)$. Its matrix entries will be denoted by $\tau_{ij}(u)$.


\subsection{Twisted Yangians as subalgebras of Yangians and extended Yangians} \label{Sec:3.2}

\begin{defn} \label{D:X(g,G):tw}
The algebra $X(\mfg,\mcG)^{tw}$ is the subalgebra of $X(\mfg)$ generated by the coefficients of the entries $s_{ij}(u)$ of the $S$-matrix $S(u)$ given by:
\eq{ \label{S=TGT}
S(u) = \begin{cases}
T(u)\,K(u)\,T(-u)^{-1}  &  \text{for type AIII} ,\\
T(u-\ka/2)\,K(u)\,T(-u+\ka/2)^t & \text{for types AI, AII $(\ka=0)$ and B--C--D $(\ka=\tfrac{N}2\mp1)$} ,
\end{cases}
}
with $T(u)\in X(\mfg)[[u^{-1}]]$ and $K(u)$ as in Section \ref{Sec:2.3}. 
\end{defn}

\begin{defn} \label{D:Y(g,G):tw}
The algebra $Y(\mfg,\mcG)^{tw}$ is the subalgebra of $Y(\mfg)$ generated by the coefficients of the entries $\si_{ij}(u)$ of the $S$-matrix $\Si(u)$ obtained from \eqref{S=TGT} by replacing $T(u)$ with $\mcT(u)$, which is equivalent to saying that
\eq{
Y(\mfg,\mcG)^{tw} = X(\mfg,\mcG)^{tw} \cap Y(\mfg) .
}
\end{defn}

\noindent The algebra $X(\mfg,\mcG)^{tw}$ (resp.\@ $Y(\mfg,\mcG)^{tw}$) of types AI and AII is the Olshanskii's twisted Yangian $Y^\pm(N)$ (resp.\@ $SY^\pm(N)$) (see \cite{Ol,MNO,Mo5}); of type AIII it is the Molev-Ragoucy reflection algebra $\mcB(N,q)$ (resp.\@ $\mcS\mcB(N,q)$) \cite{MR}; in all the remaining cases it is the extended twisted Yangian $X(\mfg_N,\mcG)^{tw}$ (resp. the twisted Yangian $Y(\mfg_N,\mcG)^{tw}$) introduced by the first two authors in \cite{GR}. \smallskip

We now briefly recall important facts about the algebras $X(\mfg,\mcG)^{tw}$ and $Y(\mfg,\mcG)^{tw}$ and their centres:

\begin{itemize}

\item Let $X(\mfg_N,\mcG)^{tw}$ be of types B--C--D. Then the $S$-matrix $S(u)$ satisfies the reflection equation \eqref{RE} with $R(u)$ given by b) in \eqref{R(u)} and $K(u)$ replaced with $S(u)$. It also satisfies the so-called symmetry relation
\eq{
S(u)^t = (\pm)\,S(\ka-u) \pm \frac{S(u)-S(\ka-u)}{2u-\ka} + \frac{\Tr(K(u))\,S(\ka-u) - \Tr(S(u))\cdot I }{2u-2\ka} \,. \label{RES}
}
The sign in $(\pm)$ is always `+' except it is `$-$' for types CI and DIII. The unitarity relation defines an even series $w(u)$, namely
\eq{
S(u)\, S(-u) = w(u)\cdot I\qu \Longrightarrow \qu w(u) = 1 + \mysum_{i\ge1} w_i u^{-i}\in X(\mfg_N,\mcG)^{tw}[[u^{-1}]] \label{Suwu}
}
with even coefficients algebraically independent and generating the centre $Z(\mfg_N,\mcG)^{tw}$ of $X(\mfg_N,\mcG)^{tw}$. The quotient algebra $X(\mfg_N,\mcG)^{tw}/(w(u)-1)$ is isomorphic to the twisted Yangian $Y(\mfg_N,\mcG)^{tw}$. 

\smallskip

\item Let $X(\mfg,\mcG)^{tw}$ be of type AI or AII, namely the algebra $Y^\pm(N)$. Then the $S$-matrix satisfies the twisted reflection equation \eqref{REt} with $R(u)$ given by a) in \eqref{R(u)} and $K(u)$ replaced with $S(u)$. It also satisfies the symmetry relation \eqref{RES} with $\ka=0$ and only the first two terms on the right-hand side. Moreover, there exists a series called the Sklyanin determinant
\eq{
{\rm sdet}\,S(u) = 1 + \mysum_{i\ge 1} c_i u^{-1} \in Y^\pm(N)[[u^{-1}]] , \label{sdetAIAII}
}
the even coefficients of which are algebraically independent and generate the centre $Z^\pm(N)$ of $Y^\pm(N)$. Introduce a rational function $\ga_N(u)$ by 
\eq{
\ga_N(u) = 1 \qu\text{for}\qu Y^+(N) \qu\text{and}\qu  \ga_N(u)=\frac{2u+1}{2u-N+1} \qu\text{for}\qu Y^-(N). \label{gamma(u)}
}
Then the quotient algebra 
\eq{ 
Y^\pm(N)/({\rm sdet}\,S(u)-\ga_N(u)) \label{quotAIAII}
}
is isomorphic to the special twisted Yangian $SY^\pm(N)$. (See \cite[Corollary 2.9.3]{Mo5}.)

\smallskip

\item Let $X(\mfg,\mcG)^{tw}$ be of type AIII, namely the algebra $\mcB(N,q)$. Then the $S$-matrix satisfies the reflection equation \eqref{RE} with $R(u)$ given by a) in \eqref{R(u)} and $K(u)$ replaced with $S(u)$, and the unitarity relation $S(u)\,S(-u)=I$. The Sklyanin determinant is given by the series
\eq{
{\rm sdet}\,S(u)=(-1)^q + \mysum_{i\ge 1} c_i u^{-i} \in \mcB(N,q)[[u^{-1}]] , \label{sdetAIII}
}
the odd coefficients of which are algebraically independent and generate the centre $\mcZ(N,q)$ of $\mcB(N,q)$. Introduce the series
\eq{
\theta(u) = (-1)^q \prod_{1\le i \le q} 2(u-N+i) \prod_{1\le i \le N-q} 2(u-N+i) \prod\limits_{1\le i\le N} \frac{1}{2u-2N+i+1}. \label{theta(u)}
} 
Then the quotient algebra 
\eq{ 
\mcB(N,q)/({\rm sdet}\,S(u) - \theta(u)) \label{quotAIII}
} 
is isomorphic to the special reflection algebra $\mcS\mcB(N,q)$.

\end{itemize}

By employing the uniform notation
\eq{
(X(\mfg,\mcG)^{tw}, \, Z(\mfg,\mcG)^{tw}, \, Y(\mfg,\mcG)^{tw} ) = 
\begin{cases} 
(Y^\pm(N), \hspace{2em} Z^\pm(N), \hspace{1.85em} SY^\pm(N) ) & \text{for types AI and AII} , \\
(\mcB(N,q), \hspace{1.78em} \mcZ(N,q), \hspace{1.65em}  \mcS\mcB(N,q) ) & \text{for type AIII} , \\
(X(\mfg_N,\mcG)^{tw}, \, Z(\mfg_N,\mcG)^{tw}, \, Y(\mfg_N,\mcG)^{tw} )\!\! & \text{for types B--C--D} ,
\end{cases} \hspace{-.5em}
}
the following isomorphism of algebras holds
\eq{
X(\mfg,\mcG)^{tw} \cong Z(\mfg,\mcG)^{tw} \ot Y(\mfg,\mcG)^{tw} . \label{TY=TZ*TSY}
}
Moreover, the algebra $Y(\mfg,\mcG)^{tw}$ has no non-trivial central elements.


\subsection{Twisted Yangians as quotients of extended reflection algebras} \label{Sec:3.3}

\begin{defn}
The extended reflection algebra $\wt{X}(\mfg,\mcG)^{tw}$ is the unital associative $\C$-algebra generated by elements $\wt s_{ij}^{(r)}$ for $-n\le i,j\le n$ and $r\in\Z_{\ge 0}$ satisfying the twisted reflection equation 
\eq{ \label{XtRE}
R(u-v)\,\wt S_1(u)\,R(-u-v)^{t}\,\wt S_2(v) = \wt S_2(v)\,R(-u-v)^{t}\,\wt S_1(u)\,R(u-v) ,
}
with $R(u)$ given by a) of \eqref{R(u)} for types AI and AII, and the reflection equation
\eq{ \label{XRE}
R(u-v)\,\wt S_1(u)\,R(u+v)\,\wt S_2(v) = \wt S_2(v)\,R(u+v)\,\wt S_1(u)\,R(u-v) ,
}
with $R(u)$ given by a) of \eqref{R(u)} for type AIII, and by b) of \eqref{R(u)} for types B--C--D. The $S$-matrix $\wt S(u)$ is defined in the usual way: $\wt S(u) = \mysum_{i,j=-n}^n \mysum_{r=0}^{\infty} E_{ij} \ot \wt s_{ij}^{(r)} u^{-r}$ with $\wt s_{ij}^{(0)} = \del_{ij}$ for types AI, AII and $\wt s_{ij}^{(0)} = g_{ij}$ for all the remaining types.
\end{defn}

\noindent The algebra $\wt{X}(\mfg,\mcG)^{tw}$ of types AI and AII is the Olshanskii's extended twisted Yangian $\wt Y^\pm(N)$ (see \cite[Section 6]{MNO} and \cite[Section 2.13]{Mo5}), of type AIII it is the Molev-Ragoucy extended reflection algebra $\wt\mcB(N,q)$ \cite{MR}, in all the remaining cases it is the extended reflection algebra $\mcX\mcB(\mcG)$ introduced in \cite{GR}. 

\smallskip

Following the same strategy as we used in the sections above, we recall some important facts about the extended reflection algebra $\wt{X}(\mfg,\mcG)^{tw}$ and its central elements:

\begin{itemize}

\item Let $\wt{X}(\mfg,\mcG)^{tw}$ be of types B--C--D, namely the algebra $\mcX\mcB(\mcG)$. Then there exists a formal power series $\msc(u) = 1 + \sum_{i\ge 1} \msc_i u^{-i}\in\mcX\mcB(\mcG)[[u^{-1}]]$ such that 
\eq{ \label{p(u)c(u)}
Q \,\wt S_1(u)\, R(2u-\ka)\, \wt S_2(\ka-u)^{-1} = \wt S_2(\ka-u)^{-1}R(2u-\ka)\, \wt S_1(u)\, Q = p(u)\, \msc(u)\, Q 
}
with the $R$-matrix given by b) in \eqref{R(u)} and
\eq{ \label{p(u)}
p(u) = (\pm)1 \mp \frac{1}{2u-\ka} + \frac{{\rm tr}(K(u))}{2u-2\ka} .
}
It can be checked that $p(u)\, p(\kappa-u) = 1- (2u-\kappa)^{-2}$. The coefficients $\msc_1, \msc_3, \ldots$ of the series $\msc(u)$ are algebraically independent and central in $\mcX\mcB(\mcG)$. Moreover, the constraint $\msc(u)=1$ is equivalent to the symmetry relation \eqref{RES} with $S(u)$ replaced by $\wt S(u)$. The quotient algebra $\mcX\mcB(\mcG)/(\msc(u)-1)$ is thus isomorphic to the twisted Yangian $X(\mfg_N,\mcG)^{tw}$: see \cite{GR}.

\smallskip

\item Let $\wt{X}(\mfg,\mcG)^{tw}$ be of type AI or AII, namely the algebra $\wt Y^\pm(N)$. Then all that was said in the paragraph above is also true, except that the $R$-matrix is given by a) in \eqref{R(u)}, $p(u)$ given by \eqref{p(u)} with $\ka=0$ and the last term set to zero, and the same for \eqref{RES}. The quotient algebra $\wt Y^\pm(N)/(\msc(u)-1)$ is isomorphic to the twisted Yangian $Y^\pm(N)$: see \cite{MNO,Mo5}.

\smallskip

\item Let $\wt{X}(\mfg,\mcG)^{tw}$ be of type AIII, namely the algebra $\wt \mcB(N,q)$. Then the unitarity relation gives 
\eq{
\wt S(u)\,\wt S(-u)=f(u)\cdot I, \label{SS=fI}
}
where $f(u) = 1 + \sum_{i\ge 1} f_i u^{-i}\in\wt \mcB(N,q)[[u^{-1}]]$ is an even series with coefficients algebraically independent and central in $\wt \mcB(N,q)$. The quotient algebra $\wt \mcB(N,q)/(f(u)-1)$ is isomorphic to the reflection algebra $\mcB(N,q)$: see \cite{MR}.

\end{itemize}

Let $\wt{Z}(\mfg,\mcG)^{tw}$ denote the commutative subalgebra of $\wt{X}(\mfg,\mcG)^{tw}$ generated by the coefficients of the even series $f(u)$ for type AIII, or by coefficients of the series $\msc(u)$ for all the remaining cases. Then the following isomorphism holds:
\eq{
\wt{X}(\mfg,\mcG)^{tw} \cong \wt Z(\mfg,\mcG)^{tw} \ot X(\mfg,\mcG)^{tw}  \cong \wt Z(\mfg,\mcG)^{tw} \ot Z(\mfg,\mcG)^{tw} \ot Y(\mfg,\mcG)^{tw} . \label{TX=TZX*TZ*TSY}
}
In particular, $\wt Z(\mfg,\mcG)^{tw} \ot Z(\mfg,\mcG)^{tw}$ is the whole centre of $\wt{X}(\mfg,\mcG)^{tw}$. The first isomorphism was not shown for types AI and AII in \cite{MNO} and for type AIII in \cite{MR}, but can be proven in the same way as in \cite{GR} for types B--C--D and as in the proof of Theorem 3.1 in \cite{AMR}. The second isomorphism follows from the first one and \eqref{TY=TZ*TSY}.


\section{Isomorphisms for small rank cases} \label{Sec:6}


\subsection{Yangians and extended Yangians of small rank} \label{Sec:4.1}

Isomorphisms between the Yangians of type A and those of types B, C, D of small rank, when $n=1$ and $n=2$, were found in \cite[Section 4]{AMR}. We briefly recall the main properties of these isomorphisms. 

\smallskip

Let $T(u)$ (resp.\@ $\mcT(u)$) denote the $T$-matrix of the extended Yangian $X(\mfg_N)$ (resp.\@ $Y(\mfg_N)$). Likewise, let $T^\circ(u)$ (resp.\ $\mcT^\circ(u)$) denote the $T$-matrix of the Yangian $Y(N)$ (resp.\ $SY(N)$). Then the maps
\eqa{
& \psi_1 \;:\; X(\mfsp_2) \to Y(2),\quad T(u) \mapsto T^{\circ}(u/2) , \label{X->Y:1} \\  
& \psi_2 \;:\; X(\mfso_3) \to Y(2),\quad T(u) \mapsto \tfrac{1}{2}R^\circ(-1)\, T^{\circ}_1(2u)\,T^{\circ}_2(2u+1) , \label{X->Y:2}
}
where $R^\circ(-1)=(I+P)$, define isomorphisms of the algebras. (See \cite[Section 4.2]{AMR} for the meaning of $T^{\circ}_1(2u)\,T^{\circ}_2(2u+1)$.) The map $\psi_1$ restricts to an isomorphism between $Y(\mfsp_2)$ and $SY(2)$ when these are viewed as subalgebras of $X(\mfsp_2)$ and $Y(2)$: see Corollary 4.2 in \textit{loc.~cit.} It also descends to an isomorphism between $Y(\mfsp_2)$ and $SY(2)$ when $Y(\mfsp_2)$ (resp.~$SY(2)$) is viewed as the quotient of $X(\mfsp_2)$ (resp.~$Y(2)$) by the ideal generated by the non-constant coefficients of the central series $z(u)$ - see \eqref{TTtz(u)} (resp.~of the quantum determinant ${\rm qdet}\, T^{\circ}(u)$ - see \eqref{qdetglN}). Indeed, starting from \eqref{TTtz(u)} and applying $\psi_1$ to it, we obtain (using \cite[Proposition 2.7]{MNO}): 
\spl{ 
\psi_1(z(u)\cdot I) = {} & T^{\circ}(u/2) \, T^{\circ }(u/2+1)^t \\ = {} & (t^{\circ}_{-1,-1}(u/2) \, t^{\circ}_{11}(u/2+1) - t^{\circ}_{-1,1}(u/2) \, t^{\circ}_{1,-1}(u/2+1) )\cdot I = {\rm qdet} \, T^{\circ}(u/2+1)\cdot  I, \label{psi1:z(u)} 
}
which implies that $\psi_1$ restricts to an isomorphism between the ideal of $X(\mfsp_2)$ generated by the non-constant coefficients of $z(u)$ and the ideal of $Y(2)$ generated by the non-constant coefficients of ${\rm qdet}\, T^{\circ}(u)$. Therefore, $\psi_1$ descends to an isomorphism between $Y(\mfsp_2)$ ($\cong X(\mfsp_2)/(z(u)-1)$) and $SY(2)$ ($\cong Y(2)/({\rm qdet}\, T^{\circ}(u)-1)$). 

It will also be useful to know what is $\psi_1(w(u))$ (see \eqref{Suwu}) in the case of $X(\mfsp_2)$. The series $w(u)$ is related to $z(u)$ via $w(u) = z(-u-1) \, z(u-1) $ (see \cite{GR}), hence 
\spl{
\psi_1(w(u)) &= \psi_1(z(-u-1)) \, \psi_1(z(u-1)) \\
&= {\rm qdet}\, T^{\circ}(-u/2+1/2) \, {\rm qdet}\, T^{\circ}(u/2+1/2) \qu\text{ by } \eqref{psi1:z(u)} \\
&= \left\lbrace \begin{array}{c} {\rm sdet}\, S^{\circ}(u/2+1/2) \text{ in } Y^+(2) \\[.25em] 
\dfrac{u}{u+2} \, {\rm sdet}\, S^{\circ}(u/2+1/2) \text{ in } Y^-(2) \end{array} \right. \hspace{.5cm}\text{ by } \eqref{Ols:qdet->sdet}. \label{psi1wu}
}

Let $\wh{\psi}_2$ be the composite of $\psi_2$ with the shift automorphism of $Y(2)$ given by $T^{\circ}(u) \mapsto T^{\circ}(u-1/2)$. Let us perform similar computations for $\wh{\psi}_2(z(u))$ and $\wh{\psi}_2(w(u))$. We need to know what is $\psi_2(T(u)^{t_+})$: it can be seen from \cite{AMR} that 
\eq{ 
\psi_2(T(u)^{t_+}) = \tfrac{1}{2} R^{\circ}(-1)\, T_1^{\circ}(2u)^{t_-} T^{\circ}_2(2u+1)^{t_-}. \label{psi2Tt} 
}   
That equals also $\frac{1}{2} R^{\circ}(-1)\,T^{\circ}_1(2u+1)^{t_-} T^{\circ}_2(2u)^{t_-}$. In the computations below, we will use the identities
$$
\left( \tfrac{1}{2} R^{\circ}(-1)\right)^2 = \tfrac{1}{2} R^{\circ}(-1), \qu \tfrac{1}{2} R^{\circ}(-1)\, R^{\circ}(2)^{t_-} = \tfrac{1}{2} R^{\circ}(-1)
$$ 
and the following RTT-relations
\eqn{
& \tfrac{1}{2} R^{\circ}(-1) \,  T^{\circ}_1(2u+1/2) \,  T^{\circ}_2(2u-1/2) = \tfrac{1}{2} R^{\circ}(-1) \,  T^{\circ}_2(2u+1/2) \,  T^{\circ}_1(2u-1/2) \\
& \qq =  T^{\circ}_2(2u+1/2) \,  T^{\circ}_1(2u-1/2) \, \tfrac{1}{2} R^{\circ}(-1) = T^{\circ}_1(2u+1/2) \,  T^{\circ}_2(2u-1/2) \, \tfrac{1}{2} R^{\circ}(-1). 
} 
The second and fourth equality follow from $\tfrac{1}{2} R^{\circ}(-1)  P = \tfrac{1}{2} R^{\circ}(-1)$. We then have 
\begin{align*}
\wh{\psi}_2(z(u)\cdot I) = {} & \wh{\psi}_2(T(u)) \, \wh{\psi}_2(T(u+ 1/2)^{t_+}) \\
 = {} & \tfrac{1}{2} R^{\circ}(-1) \, T^{\circ}_1(2u-1/2)  \, T^{\circ}_2(2u+1/2)  \, \tfrac{1}{2} R^{\circ}(-1)  \,  T^{\circ }_1(2u+3/2)^{t_-} \,  T^{\circ }_2(2u+1/2)^{t_-} \\
= {} & \tfrac{1}{2} R^{\circ}(-1) \,  T^{\circ}_1(2u+1/2) \,  T^{\circ}_2(2u-1/2)  \, \tfrac{1}{2} R^{\circ}(-1)   \, T^{\circ}_1(2u+3/2)^{t_-} \,  T^{\circ }_2(2u+1/2)^{t_-} \\
= {} & \tfrac{1}{2} R^{\circ}(-1)  \, T^{\circ}_1(2u+1/2)  \, T^{\circ}_2(2u-1/2)  \, \tfrac{1}{2} R^{\circ}(-1)  \, R^{\circ }(2)^{t_-}  \,  T^{\circ }_1(2u+3/2)^{t_-} \,  T^{\circ }_2(2u+1/2)^{t_-} \\
 = {} & \tfrac{1}{2} R^{\circ}(-1)  \, T^{\circ}_1(2u+1/2)  \, T^{\circ}_2(2u-1/2)  \, R^{\circ }(2)^{t_-}  \,  T^{\circ }_1(2u+3/2)^{t_-} \,  T^{\circ }_2(2u+1/2)^{t_-} \\
= {} & \tfrac{1}{2} R^{\circ}(-1) \, T^{\circ}_1(2u+1/2) \,  T^{\circ }_1(2u+3/2)^{t_-} \,  R^{\circ }(2)^{t_-} \,  T^{\circ}_2(2u-1/2) \,  T^{\circ }_2(2u+1/2)^{t_-} \\
= {} & \tfrac{1}{2} R^{\circ}(-1) \,  {\rm qdet}\, T^{\circ}(2u+3/2) \,  {\rm qdet} \, T^{\circ}(2u+1/2) \qu\text{ by } \eqref{psi1:z(u)}. 
\end{align*}
Since $w(u) = z(-u-1/4)\, z(u-1/4)$ in $X(\mfso_3)$, we obtain: 
\spl{
\wh{\psi}_2(w(u)) & = \wh{\psi}_2(z(-u-1/4))\, \wh{\psi}_2(z(u-1/4)) \\
& = {\rm qdet}\,T^{\circ}(-2u) \, {\rm qdet}\, T^{\circ}(-2u+1) \,  {\rm qdet}\, T^{\circ}(2u) \, {\rm qdet}\, T^{\circ}(2u+1)  \\
&  =  \left\lbrace \begin{array}{c} {\rm sdet}\,S^{\circ}(2u) \,  {\rm sdet}\,S^{\circ}(2u+1) \text{ in } Y^+(2) \\[.25em]
\dfrac{4u-1}{4u+3}\, {\rm sdet}\, S^{\circ}(2u) \, {\rm sdet}\,S^{\circ}(2u+1) \text{ in } Y^-(2) \end{array} \right. \qu \text{ by } \eqref{Ols:qdet->sdet}.  \label{psi2wu}
}

For $X(\mfso_4)$, it is known that the map
\eqa{
& \psi_3 \;:\; X(\mfso_4) \hookrightarrow Y(2)\ot Y(2),\quad\; T(u) \mapsto T^{\circ}(u)\,T^{\bullet}(u) , \label{X->Y:3}
}
is an embedding which, after composing with $Y(2) \onto SY(2)$ on either copies of $Y(2)$, leads to isomorphisms $X(\mfso_4) \iso SY(2)\ot Y(2)$ and $X(\mfso_4) \iso Y(2)\ot SY(2)$. (See Proposition 4.8 and Corollary 4.10 in \cite{AMR}.) Here $T^\circ(u)$  and $T^\bullet(u)$ denote, respectively, the $T$-matrix of the first and the second copy of the  algebra $Y(2)$ in $Y(2) \otimes Y(2)$. (The notation $T^{\circ}(u)$ and  $T^{\circ\prime}(u)$ is used instead in Section 4.3 in [AMR].) Moreover, by replacing the $T$-matrices $T(u)$ with $\mcT(u)$ (and the same for $T^\circ(u)$ and $T^\bullet(u)$) one obtains the isomorphisms \eq{ \psi_{1} : Y(\mfsp_2) \iso SY(2),\qu \psi_{2} : Y(\mfso_3) \iso SY(2)\qu \text{and} \qu \psi_{3} : Y(\mfso_4) \iso SY(2)\otimes SY(2). \label{psi2SY}} 

In the following subsections we will demonstrate analogues of the maps (\ref{X->Y:1}-\ref{X->Y:3}) for the extended twisted Yangians that correspond to the isomorphisms (\ref{iso:1}-\ref{iso:3}) between symmetric pairs.  The analogue of the isomorphism \eqref{iso:4} for twisted Yangians relates $X(\mfso_4,\mfso_3)^{tw}$ to the so-called achiral twisted Yangian; such algebras in Drinfeld original presentation were introduced by one of the authors in \cite{MaRe}. Their RTT analogues will be considered elsewhere.


\subsection{Olshanskii's twisted Yangians and Molev-Ragoucy reflection algebras of small rank} \label{Sec:4.2}

It was shown in \cite{MR} that Olshanskii's twisted Yangian $Y^\pm(2)$ (and its extensions/specializations) is isomorphic to the Molev-Ragoucy reflection algebra $\mcB(2,q)$ (and its extensions/specializations, respectively) with $q=0$ for $Y^-(2)$ and $q=1$ for $Y^+(2)$. Let us briefly recall a few facts about these algebras (see \cite{MNO} and \cite{MR}) and the isomorphism between them (see Section 4.2 in \cite{MR} for complete details). 

\smallskip

Let $S(u)$ denote the $S$-matrix of the Olshanskii's twisted Yangian $Y^\pm(2)$ and let $s_{ij}(u)$ be its matrix entries with coefficients $s^{(r)}_{ij}$ with ${i,j=\pm1}$, $r\in \Z_\geq 0$ such that $s_{ij}^{(0)}=\del_{ij}$. The $S$-matrix satisfies the twisted reflection equation \eqref{REt}, which is
\eq{
 R(u-v)\,S_1(u)\,R(-u-v)^{t}\, S_2(v) = S_2(v)\,R(-u-v)^{t}\, S_1(u)\,R(u-v), \label{Ols:RE}
}
and the symmetry relation ({\it cf.} \eqref{RES})
\eq{
S(-u)^{t} = S(u)\pm \frac{S(u)-S(-u)}{2u}\label{Ols:symm}.
}
Here $R(u) =1-u^{-1}P$ and $R^{t} =1-u^{-1}Q$. The Sklyanin determinant ${\rm sdet}\,S(u)=1 + c_1 u^{-1} + c_2 u^{-2} + \ldots$ with $c_i\in Y^\pm(2)$ can be given explicitly by the formulas 
\spl{
 {\rm sdet}\,S(u)&=\frac{2u+1}{2u\pm1}\left(s_{-1,-1}(u-1)s_{-1,-1}(-u)\mp s_{-1,1}(u-1)s_{1,-1}(-u) \right) \\
                    &=\frac{2u+1}{2u\pm1}\left(s_{11}(-u)s_{11}(u-1)\mp s_{1,-1}(-u)s_{-1,1}(u-1)\right) \label{Ols:sdet}.
}
Moreover, it satisfies the relation
\eq{
{\rm sdet}\,S(u)=
\begin{cases}
{\rm qdet}\, T(u) \, {\rm qdet}\, T(1-u) &\text{ for } Y^+(2) ,\\[.25em]
\dfrac{2u+1}{2u-1}\,{\rm qdet}\, T(u) \, {\rm qdet}\, T(1-u) &\text{ for } Y^-(2). 
\end{cases}\label{Ols:qdet->sdet}
}
The special twisted Yangian $SY^\pm(2) = Y^\pm(2) \cap SY(2)$ is isomorphic to the quotient $Y^+(2)/({\rm sdet}\, S(u)-1)$ or $Y^-(2)/({\rm sdet}\, S(u)-\frac{2u+1}{2u-1})$. Likewise, $Y^\pm(2)$ is isomorphic to the quotient $\wt Y^\pm(2)/(\msc(u)-1)$ of the extended twisted Yangian $\wt Y^\pm(2)$, where $\msc(u)$ is the formal series defined by the relation ({\it cf.}~(\ref{p(u)c(u)}-\ref{p(u)}))
\eq{
Q\,{\wt S_1}(u)\,R(2u)\,{\wt S_2}(-u)^{-1} =  \left(1\mp \frac{1}{2u}\right)\msc(u)\,Q,
}
and satisfying $\msc(u)\,\msc(-u)=1$. Here $\wt S(u)\in\wt Y^\pm(2)[[u^{-1}]]$.  

\smallskip

Let $B^\circ(u)$ denote the $S$-matrix of the Molev-Ragoucy reflection algebra $\mcB(2,q)$ and let $b^\circ_{ij}(u)$ be its matrix entries with coefficients $b_{ij}^{\circ(r)}$ with $i,j=\pm1$ and $r\ge1$ such that $b_{ij}^{\circ(0)}=\delta_{ij}\epsilon_i$ with $\epsilon_{-1}=1$ and $\epsilon_1 = \pm 1$ ($+1$ if $q=0$ and $-1$ if $q=1$). The defining relations are given by the reflection equation \eqref{RE}
\eq{
R^\circ(u-v)\,B^\circ_1(u)\,R^{\circ}(u+v)\,B^\circ_2(v) = B^\circ_2(v)\,R^{\circ}(u+v)\,B^\circ_1(u)\,R(u-v) , \label{MoRa:RE}
}
and the unitarity constraint 
\eq{
B^{\circ}(u)\,B^{\circ}(-u) = I . \label{BB=I}
}
Let $\wt B^\circ(u)$ denote the $S$-matrix of the extended reflection algebra $\wt \mcB(2,q)$. Then $\wt B^\circ(u)\,\wt B^\circ(-u)=f^\circ(u)\cdot I$; here $f^\circ(u)$ is an even series with coefficients $f^\circ_{2r}$ algebraically independent and central in $\wt \mcB(2,q)$.

An ascending filtration on $\mcB(2,q)$ can be introduced by setting $\deg b_{ij}^{\circ(r)}  = r-1$. Then the corresponding graded algebra $\gr\,\mcB(2,q)$  (resp.\ $\gr\,\mcS\mcB(2,q)$) is isomorphic to the twisted current algebra $U(\mfgl_2[x]^\rho)$ (resp.~$U(\mfsl_2[x]^\rho)$). ($\rho$ is the automorphism of $\mfgl_2$ obtained by conjugation by the diagonal matrix which is the identity ${\rm diag}(1,1)$ if $q=0$ and is $\rm{diag}(1,-1)$ if $q=1$.) The isomorphism is given by the map 
\eq{
(\epsilon_i + (-1)^{r-1}\epsilon_j)\,E_{ij}\, x^{r-1}  \mapsto  \gr\,b_{ij}^{\circ(r)} \label{Y->U}
}
and this induces $\gr\,c^\circ_{2r-1} \mapsto 2\,(-1)^q (E_{-1,-1}+E_{11})\, x^{2(r-1)}$. 

By Proposition 4.3 in \cite{MR} the mappings
\eqa{
& \phi_0 \;:\; \wt Y^-(2) \to \wt \mcB(2,0) , \qquad \wt S^\circ(u) \mapsto \wt B^\circ(u+1/2) , \label{Y->B:1} \\
& \phi_1 \;:\; \wt Y^+(2) \to \wt \mcB(2,1) , \qquad \wt S^\circ(u) \mapsto -\wt B^\circ(u+1/2)\,K , \label{Y->B:2}
}
where $K= E_{11} - E_{-1,-1}$, are algebra isomorphisms. We have similar isomorphisms between $SY^\pm(2)$ and $\mcS\mcB(2,q)$, which can be seen as follows. Let $\Si^\circ(u)$ denote the $S$-matrix of $SY^\pm(2)$ and $B^\circ(u)$ be the $S$-matrix of $\mcS\mcB(2,q)$. We have $\Si^\circ(u) = \mcT^\circ(u)\,\mcT^\circ(-u)^{t_\pm}$. We also have 
$$
\mcT^\circ(u)^{t_-} = \mcT^\circ(u-1)^{-1} , \qu  \mcT^\circ(u)^{t_+} = K\,\mcT^\circ(u-1)^{-1} K,
$$
which follow from the restriction of the isomorphism \eqref{X->Y:1} to the subalgebras $Y(\mfsp_2)$ and $SY(2)$, and the symmetry property \eqref{TTtz(u)} giving $\mcT(2u+2)^{t_-} = \mcT(2u)^{-1}$ for $\mcT(u)\in Y(\mfsp_2)[[u^{-1}]]$. Thus, the isomorphism $SY^\pm(2) \to \mcS\mcB(2,q)$ is simply given by the shift automorphism $\mcT^\circ(u) \mapsto \mcT^\circ(u+1/2)$ yielding 
\spl{
\Si^\circ(u) = \mcT^\circ(u)\,\mcT^\circ(-u)^{t_-} & \mapsto \mcT^\circ(u+1/2)\,\mcT^\circ(-u+1/2)^{t_-} \\
& = \mcT^\circ(u+1/2)\,\mcT^\circ(-u-1/2)^{-1} = B^\circ(u+1/2) \label{SYSB0}
}
if $q=0$, and 
\spl{
\Si^\circ(u) = \mcT^\circ(u)\,\mcT^\circ(-u)^{t_+} & \mapsto \mcT^\circ(u+1/2)\,\mcT^\circ(-u+1/2)^{t_-} \\
&  = \mcT^\circ(u+1/2)\,K\,\mcT^\circ(-u-1/2)^{-1}K = - B^\circ(u+1/2)\,K \label{SYSB1}
}
if $q=1$, as required.
Moreover, as remarked at the end of Section 4.2 in [MoRa], this isomorphism can then be extended to an isomorphism between $Y^\pm(2)$ and $\mcB(2,q)$ because of the tensor product decomposition~ \eqref{TY=TZ*TSY}.


\subsection{Twisted Yangians for the symmetric pairs $(\mfsp_2,\mfsp_2)$ and $(\mfsp_2,\mfgl_1)$}

The twisted Yangians for the symmetric pairs $(\mfsp_2,\mfsp_2)$ and $(\mfsp_2,\mfgl_1)$ are those of types C0 and CI with $\mcG = E_{-1,-1} + E_{11}$ and $\mcG = E_{11} - E_{-1,-1}$, respectively. We have $N=2$ and $\ka=2$. Let us set $Q^+=Q$ for the orthogonal case and $Q^-=Q$ for the symplectic case. Observe that, in the $N=2$ symplectic case, the operators $P$ and $Q^-$ satisfy the identity $P+Q^-=I \otimes I$. This implies the following relation:
\eq{ \label{sp-gl}
R(u) = \frac{u-1}{u-2}\left(I-\frac{2P}{u} \right) = \frac{u-1}{u-2}\,R^\circ(u/2) = \frac{u-1}{u}\,R^{\circ }(1-u/2)^{t_-},
}
where $R(u)$ is given by b) (for $\mfsp_2$) and $R^\circ(u)$ is given by a) in \eqref{R(u)}. We also have the identity $P + K_1 Q^+ K_1 = P + K_2\, Q^+ K_2 = I \ot I$ (where $K = E_{11} - E_{-1,-1}$) which follows from  $K_1 Q^+ K_1= Q^- = K_2 Q^+ K_2$. This implies that
\eq{ \label{so-gl}
 K_1 R(u)\,K_1 = K_2 R(u)\,K_2 = \frac{u-1}{u}\,R^{\circ}(1-u/2)^{t_+} \qu \text{and} \qu K_1 K_2 R^\circ(u) = R^\circ(u) K_1 K_2.
}

\begin{prop}
The mappings
\eqa{
& \varphi_0 \;:\; \wt X(\mfsp_2,\mfsp_2)^{tw} \to \wt \mcB(2,0) , \qquad \wt S(u) \mapsto \wt B^\circ(u/2) , \label{X->B:0} \\
& \varphi_1 \;:\; \wt X(\mfsp_2,\mfgl_1)^{tw} \,\to  \wt \mcB(2,1), \qquad \wt S(u) \mapsto -\wt B^\circ(u/2) , \label{X->B:1}
}
and
\eqa{
& \varphi'_0 \;:\; \wt X(\mfsp_2,\mfsp_2)^{tw}\to \wt Y^-(2) , \qquad \wt S(u) \mapsto \wt S^\circ(u/2-1/2) , \label{X->S:0} \\
& \varphi'_1 \;:\; \wt X(\mfsp_2,\mfgl_1)^{tw}\,\to  \wt Y^+(2), \qquad \wt S(u) \mapsto \wt S^\circ(u/2-1/2)\,K , \label{X->S:1}
}
are algebra isomorphisms.
\end{prop}

\begin{proof}
Use the first equality in \eqref{sp-gl} and compare \eqref{XRE} with \eqref{MoRa:RE}; this gives \eqref{X->B:0}. The isomorphism $\varphi_1$ can be deduced in the same way, except that a sign is needed here, which can be seen by comparing the constant terms. Using instead the second equality in \eqref{sp-gl} and comparing \eqref{XRE} with \eqref{Ols:RE}, we deduce \eqref{X->S:0}. In a similar way, the equality \eqref{so-gl} leads to \eqref{X->S:1}. One can also use the isomorphisms $\phi_0^{-1}: \wt{B}^\circ(u)\mapsto \wt{S}^\circ(u-1/2)$ from \eqref{Y->B:1} and $\phi_1^{-1}: \wt{B}^\circ(u)\mapsto -\wt{S}^\circ(u-1/2)\,K$ from \eqref{Y->B:2} to obtain \eqref{X->S:0} and \eqref{X->S:1} from \eqref{X->B:0} and \eqref{X->B:1}, respectively. 
\end{proof}

\begin{crl}
The maps $\varphi'_0$ and  $\varphi'_1$ descend to the algebras $X(\mfsp_2,\mfsp_2)^{tw}$, $X(\mfsp_2,\mfgl_1)^{tw}$ and induce isomorphisms
\eqa{
X(\mfsp_2,\mfsp_2)^{tw} &\cong Y^-(2), \hspace{.9cm} X(\mfsp_2,\mfgl_1)^{tw} \cong Y^+(2), \label{iso-sp-1} \\
Y(\mfsp_2,\mfsp_2)^{tw} &\cong SY^-(2), \qq Y(\mfsp_2,\mfgl_1)^{tw} \cong SY^+(2), \label{iso-sp-2}
}
respectively. We also have isomorphisms 
\eqa{
X(\mfsp_2,\mfsp_2)^{tw} & \cong \mcB(2,0), \hspace{.9cm} X(\mfsp_2,\mfgl_1)^{tw}  \cong \mcB(2,1) , \label{iso-sp-3} \\
Y(\mfsp_2,\mfsp_2)^{tw} & \cong \mcS\mcB(2,0), \qq Y(\mfsp_2,\mfgl_1)^{tw} \cong \mcS\mcB(2,1),\label{iso-sp-4}
}

\end{crl}

\begin{proof}
The isomorphism $X(\mfsp_2,\mfsp_2) \cong Y^-(2)$ follows from the observation that the image of the identity \eqref{p(u)c(u)} under the map $\varphi'_0$ is equivalent to the identity \eqref{Ols:symm} for $Y^-(2)$:
\eqa{
& \varphi'_0 \left( Q\,\wt{S}_1(u)\,R(2u-2)\,\wt{S}_2(2-u)^{-1} - \left(1+({2u-2})^{-1}+2({2u-4})^{-1}\right)Q \right) \el
& \qquad = \frac{2u-3}{2u-4}\left( Q^-\,S^\circ_1(u/2-1/2)\,R^\circ(u-1)\,(S^\circ_2(1/2-u/2))^{-1} - \left(1+({u-1})^{-1}\right)Q^- \right) . \nn
}
Here we used $\ka=2$ and $\Tr(K(u))=2$; the operator $Q$ on the left-hand side is the one used for $X(\mfsp_2)$, hence it equals the operator $Q^-$ in the notation used for $Y^-(2)$. We also had to use \eqref{sp-gl}. The previous equality shows that $\varphi_0'$ sends $p(u) (\msc(u)-1)\,Q$ to  $( 1 + ({u-1})^{-1}) (\msc^\circ(u/2-1/2)-1)\, Q^-$, hence it establishes an isomorphism between the ideal of $\wt X(\mfsp_2,\mfsp_2)^{tw}$ generated by $\msc_i, i\ge 1$, and the ideal of $\wt{Y}^-(2)$ generated by the non-constant coefficients of $\msc^\circ(u)$. Since $Y^-(2)$ is isomorphic to the quotient of  $\wt{Y}^-(2)$ by this ideal, it follows that  $Y^-(2)$ is isomorphic to $\wt X(\mfsp_2,\mfsp_2)^{tw}/(\msc_i, i\ge 1)$, hence to $X(\mfsp_2,\mfsp_2)^{tw}$.

Similarly, for $\wt{X}(\mfsp_2,\mfgl_1)^{tw}$, we have $\Tr(K(u))=0$ giving
\eqa{
& \varphi'_1 \left( Q\,S_1(u)\,R(2u-2)\,S_2(2-u)^{-1} - \left(-1+({2u-2})^{-1}\right)Q \right) \el
& \qquad = \frac{2u-3}{2u-4}\left( Q^-\,S^\circ_1(u/2-1/2)\,K_1\,R^\circ(u-1)\,K_2\,S^\circ_2(1/2-u/2)^{-1} + \left(1-({u-1})^{-1}\right) Q^- \right)  \el
& \qquad = \frac{2u-3}{2u-4}\,K_2\left( Q^+\,S^\circ_1(u/2-1/2)\,R^\circ(u-1)\,S^\circ_2(1/2-u/2)^{-1} - \left(1-({u-1})^{-1}\right) Q^+ \right)K_1 , \nn
}
where in the second equality we have used the identities $Q^- = K_2 Q^+ K_2$ and $Q^- = - K_2 Q^+ K_1$. This shows that $\varphi_1'$ establishes an isomorphism between the ideal of $\wt X(\mfsp_2,\mfgl_1)^{tw}$ generated by $\msc_i, i\ge 1$, and the ideal of $\wt{Y}^+(2)$ generated by the non-constant coefficients of $\msc^{\circ}(u)$; consequently, $X(\mfsp_2,\mfgl_1)^{tw}$ is isomorphic to $Y^+(2)$. The isomorphisms in \eqref{iso-sp-1} have thus been established.

Let $g(u)$ be an even power series with constant coefficient $1$ and set $\check{g}(u) = g(2u)$. $Y(\mfsp_2,\mfsp_2)^{tw}$ is isomorphic to the subalgebra of $X(\mfsp_2,\mfsp_2)^{tw}$ fixed by all the automorphisms $\nu_g$ (see \cite{GR}, Section 3.2) and $SY^-(2)$ is isomorphic to the subalgebra of $Y^-(2)$ fixed by the similar automorphisms $S^{\circ}(u) \mapsto g(u)S^{\circ}(u)$ (or with $g(u)$ replaced by $\check{g}(u)$), which we also denote by $\nu_g$. In $X(\mfsp_2,\mfsp_2)^{tw}$, $\nu_g(S(u)) = g(u-1)S(u)$, so it follows from \eqref{X->S:0} that $\varphi'_0 \circ \nu_g = \nu_{\check{g}} \circ \varphi'_0$. This implies that $\varphi'_0$ restricts to an isomorphisms between $Y(\mfsp_2,\mfsp_2)^{tw}$ and $SY^-(2)$.  The same argument works for $\varphi'_1$.  

Furthermore, $\varphi'_1(w(u)) = {\rm sdet}\, S^{\circ}(u/2 + 1/2)$ in $Y^+(2)$ by \eqref{Ols:sdet}, so $\varphi'_1$ descends to an isomorphism between $X(\mfsp_2,\mfgl_1)^{tw}/(w(u)-1)$ ($\cong Y(\mfsp_2,\mfgl_1)^{tw}$) and $Y^+(2)/({\rm sdet}\, S^{\circ}(u)-1)$ ($\cong Y^+(2)$): the ideal generated by the non-constant coefficients of ${\rm sdet}\, S^{\circ}(u)$ is the same as the ideal generated by the non-constant coefficients of ${\rm sdet}\, S^{\circ}(u/2+1/2)$.  The same is true for $\varphi'_0$ except that one needs to take into account the extra factor $\gamma_2(u)$: see \eqref{gamma(u)}. 

The remaining isomorphisms $X(\mfsp_2,\mfsp_2)^{tw} \cong \mcB(2,0)$, $X(\mfsp_2,\mfgl_1)^{tw} \cong \mcB(2,1)$ and $Y(\mfsp_2,\mfsp_2)^{tw} \cong \mcS\mcB(2,0)$, $Y(\mfsp_2,\mfgl_1)^{tw} \cong \mcS\mcB(2,1)$ follow from the isomorphisms \eqref{SYSB0} and \eqref{SYSB1}.

Another way to obtain that $X(\mfsp_2,\mfsp_2)^{tw}$ and $Y^-(2)$ are isomorphic would be to view both of these as subalgebras of $X(\mfsp_2)$ and $Y(2)$ and restrict to those twisted Yangians the isomorphism $\psi_1$ of \eqref{X->Y:1}. Using the isomorphism between $Y(\mfsp_2)$ and $SY(2)$ and viewing $Y(\mfsp_2,\mfsp_2)^{tw}$ and $\mcS\mcB(2,0)$ as subalgebras of those Yangians, it is possible to see that $Y(\mfsp_2,\mfsp_2)^{tw}$ and $\mcS\mcB(2,0)$ are isomorphic. It follows from the tensor product decomposition \ref{TY=TZ*TSY} that $X(\mfsp_2,\mfsp_2)^{tw}$ and $\mcB(2,0)$ are also isomorphic. The same works for $Y(\mfsp_2,\mfgl_1)^{tw}$ and $\mcS\mcB(2,1)$.
\end{proof}


\subsection{Twisted Yangians for the symmetric pairs $(\mfso_3,\mfso_3)$ and $(\mfso_3,\mfso_2)$}

We recall the presentation of $X(\mfso_3)$ and the notation given in Section 4.2 in \cite{AMR}. Let the standard basis for the vector space $\C^2$ be given by vectors $e_{-1}$ and $e_1$. Set $V$ to be the three--dimensional subspace of $\C^2\ot\C^2$ spanned by vectors $v_{-1}=e_{-1}\ot e_{-1}$, $v_{0}=\tfrac{1}{\sqrt{2}} (e_{-1}\ot e_1+e_1\ot e_{-1})$, $v_{1}=-e_1\ot e_1$. Introduce operators $P_V, Q_V\in \End(V\ot V)$ in the usual way, as in \eqref{PQ}. Then the defining relation \eqref{RTT} for $X(\mfso_3)$ may be written as
$$
R_V (u-v)\,T_{V,1}(u)\,T_{V,2}(v) =  T_{V,2}(v)\,T_{V,1}(u)\,R_V (u-v) , 
$$
with $T_{V,1}(u),\,T_{V,2}(u)\in \End\,V\ot \End\,V\ot X(\mfso_3)[[u^{-1}]]$ and where (see \cite[Lemma 4.5]{AMR}) 
$$
R_V(u) = \dfrac{2u-1}{2u+1}\left(I_V-\frac 1u\,{P_V}+\frac{1}{u-1/2}\,{Q_V} \right) .
$$
$R_V(u)$ is, up to a power series, the $R$-matrix used previously to define $X(\mfso_3)$. Moreover, the matrix $\tfrac{1}{2}R^\circ(-1)=\tfrac{1}{2}(I+P) \in \End(\C^2\ot\C^2)$ is a projector of $\C^2\ot\C^2$ to the subspace $V$.

\smallskip

Let $A=${\small$\begin{pmatrix} \frac{1}{2} & \frac{i}{\sqrt{2}} & \frac{1}{2} \\ -\frac{i}{\sqrt{2}} & 0 & \frac{i}{\sqrt{2}} \\ \frac{1}{2} & -\frac{i}{\sqrt{2}} & \frac{1}{2} \end{pmatrix}$} and set  $\mcG'= A${\small$\begin{pmatrix} 0 & 0 & 1 \\ 0 & 1 & 0 \\ 1 & 0 & 0 \end{pmatrix}$}$ A^t = ${\small$\begin{pmatrix} 1 & 0 & 0 \\ 0 & -1 & 0 \\ 0 & 0 & 1 \end{pmatrix}$}. By \cite[Remark 3.2]{GR}, the matrix $\mcG'$ can be used instead of $\mcG=${\small$\begin{pmatrix} 0 & 0 & 1 \\ 0 & 1 & 0 \\ 1 & 0 & 0 \end{pmatrix}$} to define $\wt X(\mfso_3,\mfso_2)^{tw}$, which is what we will do for the remainder of this section. Moreover, we have that $K'(u) = \dfrac{I - 4u\, \mcG'}{1-4u} = ${\small$ \begin{pmatrix} 1 & 0 & 0 \\ 0 & \dfrac{1+4u}{1-4u} & 0 \\ 0 & 0 & 1 \end{pmatrix}$}. 

\smallskip

Let $( v_i\, |\, v_j ) = \delta_{ij}$ denote the invariant bilinear form on $V$. For any $A(u) \in \End\, V[[u^{-1}]]$, we have $A_{ij}(u) = ( v_i \,|\, A(u)\, v_j )$, where $v_i$ are basis vectors of $V$.  Moreover, let $\langle e_a\ot e_b\,|\,e_c\ot e_d\rangle=\delta_{ac}\delta_{bd}$ be the standard bilinear form on $\C^2\ot\C^2$. Then let $A^{\circ}(u) \in \End(\C^2\ot\C^2)[[u^{-1}]]$ be such that  $A_{ij}(u) = \langle{ v_i \, |\, A^\circ(u) \, v_j \rangle}$, where now $v_{-1}= e_{-1}\ot e_{-1}$, $v_{0}= \frac{1}{\sqrt{2}}(e_{-1}\ot e_1+e_1\ot e_{-1})$, $v_{1} = -e_1 \ot e_1$. For the matrix $K'(u)$ above, $K^{\prime\circ}(u)$ can be given by  
\eq{
K^{\prime\circ}(u)=\text{\small$\begin{pmatrix} 1 & 0 & 0 & 0 \\ 0 & \frac{1+4u}{1-4u} & 0 & 0 \\ 0 & 0 & \frac{1+4u}{1-4u} & 0 \\ 0 & 0 & 0 & 1 \end{pmatrix}$}. \label{K-image}
} 
If we view $K^{\prime\circ}(u)$ as an element of $\End(\C^2)^{\ot2}[[u^{-1}]]$, $K^{\prime\circ}(u)$ as in $\End\, V[[u^{-1}]]$ and $\tfrac12 R^\circ(-1)$ as an element of $\Hom(\C^2\ot\C^2,V)$, then $K'(u)$ equals the restriction of $\tfrac12 R^\circ(-1) K^{\prime\circ}(u)$ to the subspace $V$.

\pagebreak
  
\begin{prop} \label{P:so3:1}
Let $q=0$ if $\pm =-$ and $q=1$ if $\pm=+$. We then have the following: 
\begin{enumerate} [itemsep=0.5ex]

\item \label{psi21} The restriction of $\wh{\psi}_2$ to $X(\mfso_3,\mfso_{3-q})^{tw}$ yields an  isomorphism with $Y^{\pm}(2)$ given by  
$$
\wh{\psi}_2(S(u)) = \tfrac12 R^\circ(-1)\, S_1^\circ(2u-1)\, R^{\circ}(-4u+1)^{t_\pm} S_2^\circ(2u)\, (K_1 K_2)^q.
$$

\item \label{psi22} The restriction of $\wh{\psi}_2$ to $Y(\mfso_3,\mfso_{3-q})^{tw}$, viewed as the quotient of $X(\mfso_3,\mfso_{3-q})^{tw}$ by the ideal generated by the coefficients of the unitarity relation, descends to an isomorphism with $SY^\pm(2)$. 

\item \label{psi23} The restriction of $\wh{\psi}_2$ to $Y(\mfso_3,\mfso_{3-q})^{tw}$, viewed this time as a subalgebra of $X(\mfso_3,\mfso_{3-q})^{tw}$, also yields an isomorphism with $SY^\pm(2)$.  

\end{enumerate}
\end{prop}

\begin{proof} 
We first provide a proof for $X(\mfso_3,\mfso_3)^{tw}$. We will explain afterwards the modifications necessary for $X(\mfso_3,\mfso_2)^{tw}$. \smallskip

\noindent Proof of \eqref{psi21}: 
We view $X(\mfso_3,\mfso_3)^{tw}$ as a subalgebra of $X(\mfso_3)$ defined by $S(u) = T(u-1/4)\,T(-u+1/4)^{t_+}$. This allows us to compute $\wh{\psi}_2(S(u))$. It can be checked directly that $\tfrac12 R^\circ(-1)\, R^{\circ}(-4u+1)^{t_-} = \tfrac12 R^\circ(-1) $, which we need to obtain the third equality below:
\begin{align}
\wh{\psi}_2(S(u)) = {} & \wh{\psi}_2(T(u-1/4)\, T(-u+1/4)^{t_+}) = \wh{\psi}_2(T(u-1/4)) \,  \wh{\psi}_2(T(-u+1/4)^{t_+}) \notag \\
= {} &  \tfrac12 R^\circ(-1)\, T^\circ_1(2u-1)\, T^\circ_2(2u)\, \tfrac12 R^\circ(-1) \, T^{\circ}_1(-2u+1)^{t_-} T^{\circ}_2(-2u)^{t_-} \qu \text{ by } \eqref{psi2Tt}; \notag \\
= {} &  \tfrac12 R^\circ(-1)\, T^\circ_1(2u-1)\, T^\circ_2(2u)\, \tfrac12 R^\circ(-1) \,R^{\circ}(-4u+1)^{t_-} T^{\circ}_1(-2u+1)^{t_-}  T^{\circ}_2(-2u)^{t_-}   \notag \\ 
= {} &  \tfrac12 R^\circ(-1)\, T^\circ_1(2u-1)  \, T^{\circ}_1(-2u+1)^{t_-} R^{\circ}(-4u+1)^{t_-} T^\circ_2(2u) \, T^{\circ}_2(-2u)^{t_-}  \notag  \\ 
= {} & \tfrac12 R^\circ(-1)\, S^{\circ}_1(2u-1) \, R^{\circ}(-4u+1)^{t_-}  S^{\circ}_2(2u). \label{psi2Su1}
\end{align}
This shows that the image of the restriction of $\wh{\psi}_2$ to $X(\mfso_3,\mfso_3)^{tw}$ is contained in $Y^-(2)$. Surjectivity onto $Y^-(2)$ can be checked as in the proof of Proposition \ref{P:so3so3so2} given below. Since $\psi_2$ is an isomorphism, this restriction of $\wh{\psi}_2$ is also an isomorphism. 

Similar computations work for the restriction of $\wh{\psi}_2$ to $X(\mfso_3,\mfso_2)^{tw}$, except that the following changes must be taken into consideration. We now have $S(u) = T(u-1/4)\, K'(u)\, T(-u+1/4)^{t_+}$, where $K'(u) = (I - 4u\mcG')(1-4u)^{-1}$ with $\mcG'$ given above. We will also make use of the following identities: 
\spl{
& \tfrac12 R^\circ(-1) \, R^\circ(-4u+1)^{t_+}\,  K_1\,  K_2 = \tfrac12 R^\circ(-1)\,  K^{\prime\circ}(u) = K^{\prime\circ}(u)\,  \tfrac12 R^\circ(-1) , \\
& K_1 \, R(-4u+1)^{t_+} K_1 = \tfrac{4u}{4u-1}R(4u) ,\\
& K \, T^\circ(u)^{t_-}  K = T^\circ(u)^{t_+} , \label{3ids}
}
Similar computations to those in \eqref{psi2Su1} can now be performed that in addition require the following consequence of the RTT-relation: 
\[ 
T_2^\circ(2u) R^\circ(-4u+1)^{t_+}  T_1^\circ(-2u+1)^{t_+} =  T_1^\circ(-2u+1)^{t_+}  R^{\circ}(-4u+1)^{t_+} T_2^\circ(2u). 
\] 
Then the final result is that 
\eq{
\wh{\psi}_2(S(u)) =  \tfrac12 R^\circ(-1) \,  S_1^\circ(2u-1) \,  R^{\circ}(-4u+1)^{t_+}  S_2^\circ(2u)\, K_1 K_2. \label{psi2Su2}
}

\noindent Proof of \eqref{psi22}:  As follows from \eqref{psi2wu}, $\wh{\psi}_2$ maps the ideal of $X(\mfso_3,\mfso_3)^{tw}$ generated by the non-constant coefficients of $w(u)$ bijectively to the ideal of $Y^-(2)$ generated by the non-constant coefficients of $\gamma_2(u)^{-1}{\rm sdet}\,S^{\circ}(u)$. Therefore, $\wh{\psi}_2$ descends to an isomorphism between the quotient of $X(\mfso_3,\mfso_3)^{tw}$ by the first ideal (which is $Y(\mfso_3,\mfso_3)^{tw}$) and the quotient of $Y^-(2)$ by the second ideal (which is $SY^-(2)$). Exactly the same argument can be applied to $Y(\mfso_3,\mfso_2)^{tw}$ and $Y^+(2)$. \smallskip

\noindent Proof of \eqref{psi23}: Let $f(u)$ be an even power series in $u^{-1}$ with constant coefficient equal to $1$ and $\nu_f$ the automorphism of $X(\mfso_3,\mfso_3)^{tw}$ given by $S(u) \mapsto f(u-1/4) S(u)$. The series $f(u)$ can be factored as $f(u-1/4) = g(u-1/4)\,g(-u+1/4)$ and $g(u-1/4)$ can further be factored as $g(u-1/4) = h(2(u-1/4)+1/2)\, h(2(u-1/4)-1/2)$. (Both $g(u)$ and $h(u)$ can be taken to have constant term equal to $1$.) Set $h^{\circ}(u) = h(u)\,h(-u)$, so $f(u-1/4) = h^{\circ}(2u)\, h^{\circ}(2u-1)$. Then $\wh{\psi}_2 \circ \nu_f = \nu_{h^{\circ}} \circ \wh{\psi}_2$. This show that $\wh{\psi}_2$ establishes an isomorphism between the subalgebra of $X(\mfso_3,\mfso_3)^{tw}$ fixed by all the automorphisms $\nu_f$ (which is $Y(\mfso_3,\mfso_3)^{tw}$) and the subalgebra of $Y^-(2)$ fixed by all the automorphisms $\nu_{h^{\circ}(u)}$ with $h^{\circ}(u)$ even (which is $SY^-(2)$). Exactly the same argument works for $Y(\mfso_3,\mfso_2)^{tw}$ and $Y^+(2)$.
\end{proof}

Another way to check that the restriction of $\wh{\psi}_2$ to $X(\mfso_3,\mfso_3)^{tw}$ descends to an isomorphism between $Y(\mfso_3,\mfso_3)^{tw}$ and $SY^-(2)$ is to verify directly that $\Si^{\circ}(u)=\tfrac12 R^\circ(-1)\, \Sigma^{\circ}_1(2u-1) \, R^{\circ}(-4u+1)^{t_-}  \, \Sigma^{\circ}_2(2u)$ satisfies the unitarity condition $\Si^\circ(u)\,\Si^\circ(-u)=I$. Indeed, we have that $\Si^\circ(u)\,\Si^\circ(-u)$ is equal to 
\begin{align*}
 & {} \tfrac12 R^\circ(-1)\, \Sigma^{\circ}_1(2u-1) \, R^{\circ}(-4u+1)^{t_-} \Sigma^{\circ}_2(2u)\,  \tfrac12 R^\circ(-1)\, \Sigma^{\circ}_1(-2u-1) \, R^{\circ}(1+4u)^{t_-}  \Sigma^{\circ}_2(-2u) \\
& {} = \tfrac12 R^\circ(-1)\, \Sigma^{\circ}_1(2u-1) \, R^{\circ}(-4u+1)^{t_-} \Sigma^{\circ}_2(2u) \, \tfrac12 R^\circ(-1)\, \Sigma^{\circ}_1(-2u-1) \, R^{\circ}(1+4u)^{t_-}  \Sigma^{\circ}_2(-2u) \, \tfrac12 R^\circ(-1) \\
& {} \qq \text{ by the reflection equation and } ( \tfrac12 R^\circ(-1))^2 =  \tfrac12 R^\circ(-1); \\
& {} = \tfrac12 R^\circ(-1)\, \Sigma^{\circ}_1(2u-1) \, R^{\circ}(-4u+1)^{t_-} \Sigma^{\circ}_2(2u) \,\tfrac12 R^\circ(-1)  \Sigma^{\circ}_2(-2u-1) \, R^{\circ}(1+4u)^{t_-}  \Sigma^{\circ}_1(-2u) \, \tfrac12 R^\circ(-1) \\
& {} = \tfrac12 R^\circ(-1)\, \Sigma^{\circ}_1(2u-1) \, R^{\circ}(-4u+1)^{t_-} \Sigma^{\circ}_2(2u) \,  \Sigma^{\circ}_2(-2u-1) \, R^{\circ}(1+4u)^{t_-}  \Sigma^{\circ}_1(-2u) \, \tfrac12 R^\circ(-1) \\
& {} \qq \text{ after using twice the reflection equation;} \\
& {} = \tfrac12 R^\circ(-1)\, \Sigma^{\circ}_1(2u-1) \, R^{\circ}(-4u+1)^{t_-}  R^{\circ}(4u+1)^{t_-}  \Si^{\circ}_1(-2u) \, \tfrac12 R^\circ(-1) \\
& {}  \qq \text{ because } \Si^{\circ}_2(2u)\, \Sigma^{\circ}_2(-2u-1) = I, \text{ a consequence of \eqref{SYSB0} and \eqref{BB=I} }; \\
& {} = \tfrac12 R^\circ(-1)\, \Sigma^{\circ}_1(2u-1) \, \Sigma^{\circ}_1(-2u) \, \tfrac12 R^\circ(-1) \\
& {} \qq \text{ since } R^{\circ}(-4u+1)^{t_-} R^{\circ}(4u+1)^{t_-} =I; \\
& {} = \tfrac12 R^\circ(-1).
\end{align*}
Very similar computations work for $Y(\mfso_3,\mfso_2)^{tw}$. \smallskip

It is also possible to establish isomorphisms between the twisted Yangians $Y(\mfso_3,\mfso_{3-q})^{tw}$ and reflection algebras $\mcS\mcB(2,q)$.

\begin{prop}\label{Yso3so2pB2p}
We have isomorphisms $Y(\mfso_3,\mfso_{3-q})^{tw} \cong \mcS\mcB(2,q)$. Moreover, these can be extended to isomorphisms $X(\mfso_3,\mfso_{3-q})^{tw} \cong \mcB(2,q)$.
\end{prop}

\begin{proof}
The map $\psi_2$ gives an isomorphism between $Y(\mfso_3)$ and $SY(2)$ (see \eqref{psi2SY}). It can be restricted to the subalgebra $Y(\mfso_3,\mfso_{3-q})^{tw}$ of $Y(\mfso_3)$ and computations similar to those in \eqref{psi2Su1} can be performed. The difference is that $\mcT(-u-1/4)^{-1}$ is used instead of the transpose $\mcT(-u+1/4)^{t_+}$ giving
\eq{
\psi_2(\Si(u)) =  \tfrac{4u}{4u-1} \,\tfrac12 R^\circ(-1)\, B^\circ_1(2u-\tfrac12)\, R^\circ(4u)\, B^\circ_2(2u+\tfrac12), \label{psi2:SB}
}
where $B^\circ(u) \in \mcS\mcB(2,q)[[u^{-1}]]$. Here, we used the relations $\tfrac12 R^\circ(4u) R^\circ(-1)  = \tfrac{4u-1}{4u}\tfrac12 R^\circ(-1)$, \eqref{3ids} and the fact that, in $Y(\mfso_3)$, $\mcT(u)^{t_+} = \mcT(u-1/2)^{-1}$ and the inverse of $\psi_2(\mcT(u))$ is $\tfrac12 R^\circ(-1)\,\mcT^\circ_2(2u+1)^{-1} \mcT^\circ_1(2u)^{-1}$.  One difference between this and \eqref{psi2Su1} if $q=0$ or \eqref{psi2Su2} if $q=1$ is the factor $\frac{4u}{4u-1}$: this is explained by the fact that $R^{\circ}(-4u+1)^{t_-} = \frac{4u}{4u-1}\,  R^{\circ}(4u)$ or the middle identity in \eqref{3ids}, respectively. Thus it follows that $Y(\mfso_3,\mfso_{3-q})^{tw}$ is isomorphic to $\mcS\mcB(2,q)$. This isomorphism can be extended to an isomorphism between $X(\mfso_3,\mfso_{3-q})^{tw}$ and $\mcB(2,q)$ using the tensor product decompositions \eqref{TY=TZ*TSY} and the fact that the centres of both algebras are polynomial rings in infinitely many variables. 
\end{proof}


It is a natural question to ask if the isomorphisms in the previous two propositions also hold at the level of extended twisted Yangians. This is indeed the case.

\begin{prop}\label{P:so3lift}
The map
\eq{
\wt\psi_2 \;:\; \wt X(\mfso_3,\mfso_{3-q})^{tw} \to \wt Y^\pm(2), \qu \wt S(u) \mapsto \tfrac12 R^\circ(-1) \wt S^\circ_1(2u-1) R^\circ(-4u+1)^{t_\pm} \wt S^\circ_2(2u) (K_1 K_2)^q , \label{psi2lift}
}
viewed as an extension of the map  in Proposition \ref{P:so3:1} \eqref{psi21}, is an isomorphism of algebras.
\end{prop}

\begin{proof}
We use  \eqref{TX=TZX*TZ*TSY}: here, $X(\mfso_3,\mfso_{3-q})^{tw}$ is identified with the subalgebra of $\wt X(\mfso_3,\mfso_{3-q})^{tw}$ generated by the coefficients of $\wt{S}(u) \msv(u)^{-1}$ where the central power series $\msv(u)$ is such that $\msc(u) = \msv(u)^2$ and the constant term of $\msv(u)$ is one (see \cite{GR}, Theorem 5.3). We also need that decomposition for $\wt{Y}^{\pm}(2)$ and corresponding power series $\msc^{\circ}(u), \msv^{\circ}(u)$ and $\wt{S}^{\circ}(u) \msv^{\circ}(u)^{-1}$. Thus we have that
$$
\wt\psi_2(\wt S(u)) = \wt\psi_2(\msv(u))\,\wt\psi_2(S(u)) = \msv^\circ(2u-1)\, \msv^\circ(2u)\, \tfrac12 R^\circ(-1) S^\circ_1(2u-1) R^\circ(-4u+1)^{t_\mp} S^\circ_2(2u) (K_1 K_2)^q.
$$
Since $\wt\psi_2(S(u))$ satisfies the reflection equation by Proposition \ref{P:so3:1} and $\wt\psi_2(\msv(u))$ is central, it follows that $\wt\psi_2(\wt S(u))$ also satisfies the reflection equation, hence $\wt\psi_2$ is an algebra homomorphism. From the above formula, we deduce that $\wt\psi_2(\msv(u)) = \msv^\circ(2u-1)\, \msv^\circ(2u)$ and it follows that $\wt\psi_2(\msc(u)) = \msc^{\circ}(2u-1) \,\msc^{\circ}(2u)$.  

To complete the proof that $\wt\psi_2$ is an isomorphism, due to the decomposition in \eqref{TX=TZX*TZ*TSY} and Proposition \ref{P:so3:1}, it is enough to see that it is an isomorphism when restricted to the subalgebra of $\wt X(\mfso_3,\mfso_{3-q})^{tw}$ generated by the coefficients of $\msc(u)$. It was proved in \cite{GR} that this subalgebra is a polynomial algebra generated by the odd coefficients of $\msc(u)$ and the same is true for $\msc^{\circ}(u)$ and $\wt{Y}^{\pm}(2)$. We thus have to see that the map given by $\wt\psi_2(\msc(u)) = \msc^{\circ}(2u-1)\, \msc^{\circ}(2u)$ is an isomorphism between $\C[\msc_1,\msc_3, \ldots]$ and $\C[\msc^{\circ}_1,\msc^{\circ}_3, \ldots]$. This is indeed the case because $\wt\psi_2(\msc_{2i+1}) = 2^{-2i} \msc_{2i+1}^{\circ} + A(\msc^\circ_{1},\ldots,\msc^\circ_{2i-1})$ for all $i\ge0$, where $A(\msc^\circ_{1},\ldots,\msc^\circ_{2i-1})$ is a polynomial expression.
\end{proof}

\begin{crl}
The map
\eq{
\wt X(\mfso_3,\mfso_{3-q})^{tw} \to \wt \mcB(2,q), \qu \wt S(u) \mapsto \tfrac{4u}{4u-1}\tfrac12 R^\circ(-1)\, \wt B^\circ_1(2u-\tfrac12)\, R^\circ(4u)\, \wt B^\circ_2(2u+\tfrac12) ,
}
is an isomorphism of algebras.
\end{crl}

\begin{proof}
This follows by composing the isomorphism \eqref{psi2lift} with the isomorphism given in (\ref{Y->B:1}-\ref{Y->B:2}).
\end{proof}


The isomorphism between the extended reflection algebras $\wt X(\mfso_3,\mfso_{3-q})^{tw}$ and $\wt\mcB(2,q)$ (and consequently between  $\wt X(\mfso_3,\mfso_{3-q})^{tw}$ and $\wt Y^\pm(2)$) can also be established in a direct way. For this we need some additional identities. 
Recall that the matrix $\tfrac{1}{2}R^\circ(-1)$ is a projector of $\C^2\ot\C^2$ to the subspace $V$. Hence the matrix $R_V(u)$ is the restriction to $V \otimes V$ of the element of $\End(\C^2)^{\otimes 4}$ given by (see proof of \mbox{\cite[Proposition 4.4]{AMR}})
\eqa{
R_V(u) &=  \tfrac{1}{4} R^\circ_{12}(-1)\,R^\circ_{34}(-1)\,R^\circ_{14}(2u-1)\,R^\circ_{13}(2u)\,R^\circ_{24}(2u)\,R^\circ_{23}(2u+1)  \label{RV:0} \\
& =  \tfrac{1}{4} R^\circ_{23}(2u+1)\,R^\circ_{24}(2u)\,R^\circ_{13}(2u)\,R^\circ_{14}(2u-1)\,R^\circ_{12}(-1)\,R^\circ_{34}(-1). \label{RV:1}
}
Here the second line follows by application of the Yang-Baxter equation \eqref{YBE}. Moreover, we have the identity
\eq{ \label{RRR=R}
\tfrac{1}{4} R_V(u)\,R^\circ_{12}(-1)\,R^\circ_{34}(-1) = R_V(u) =\tfrac{1}{4} R_{12}^{\circ}(-1) R_{34}^{\circ}(-1) R_V(u).
}
By inserting the identity matrix $I^\circ=P^\circ_{34}P^\circ_{34}$ and using \eqref{RV:1} we can rewrite $R_V(u)$ as
\eqa{ \label{RV:2}
R_V(u) &= \tfrac{1}{16} R^\circ_{12}(-1)\,R^\circ_{34}(-1)\,P^\circ_{34}\,P^\circ_{34}\,R^\circ_{14}(2u-1)\,R^\circ_{13}(2u)\,R^\circ_{24}(2u)\,R^\circ_{23}(2u+1)\,P^\circ_{34}\,P^\circ_{34}\,R^\circ_{12}(-1)\,R^\circ_{34}(-1) \el
& = \tfrac{1}{16} R^\circ_{12}(-1)\,R^\circ_{34}(-1)\,P^\circ_{34}\,R^\circ_{13}(2u-1)\,R^\circ_{14}(2u)\,R^\circ_{23}(2u)\,R^\circ_{24}(2u+1)\,P^\circ_{34}\,R^\circ_{12}(-1)\,R^\circ_{34}(-1) \el
& = \tfrac{1}{4} R^\circ_{12}(-1)\,R^\circ_{34}(-1)\,R^\circ_{13}(2u-1)\,R^\circ_{14}(2u)\,R^\circ_{23}(2u)\,R^\circ_{24}(2u+1) ,
}
where in the last step we have used the relation $R^\circ_{34}(-1)\,P^\circ_{34} = R^\circ_{34}(-1)$ and the Yang-Baxter equation. In a similar way, using $I^\circ={P^\circ_{12}P^\circ_{12}}$, we find another form of $R_V(u)$, namely
\spl{ \label{RV:3}
R_V(u) &= \tfrac{1}{4} R^\circ_{12}(-1)\,R^\circ_{34}(-1)\, R^\circ_{23}(2u-1)\,R^\circ_{24}(2u)\,R^\circ_{13}(2u)\,R^\circ_{14}(2u+1) ,
}
which will be used in proving Proposition \ref{P:so3so3so2} below.

\begin{prop}\label{P:so3so3so2} 
We have algebra isomorphisms
\eqa{
& \varphi_q \;:\; \wt X(\mfso_3,\mfso_{3-q})^{tw} \to \wt \mcB(2,q) , \quad \wt S(u) \mapsto \tfrac{4u}{4u-1} \tfrac{1}{2} R^\circ_{12}(-1)\,\wt B^\circ_1(2u-1/2)\,R^\circ_{12}(4u)\, \wt B^\circ_2(2u+1/2) , \label{P62:1} \\
& \varphi'_q \;:\; \wt X(\mfso_3,\mfso_{3-q})^{tw} \to \wt Y^{\pm}(2) , \quad \wt S(u) \mapsto \tfrac{1}{2} R^\circ_{12}(-1)\, \wt S^\circ_1(2u-1)\,R^\circ_{12}(-4u+1)^{t_{\pm}}\, \wt S^\circ_2(2u) (K_1 \, K_2)^q. \label{P62:3} 
}
\end{prop}

\begin{proof}
We first prove that $\varphi_q$ is a surjective algebra homomorphism; it will follow from \eqref{Y->B:1} and \eqref{Y->B:2} that $\varphi'_q$ is as well. The proof of injectivity for $\varphi'_q$ is identical to the one given in Proposition \ref{P:so3lift} and injectivity of $\varphi_q$ then follows from $\varphi_q = \phi_q \circ \varphi'_q$. 

Let us start by verifying that the map \eqref{P62:1} is an algebra homomorphism. We need to check that the reflection equation
\eq{ \label{RE:V}
R_V(u-v)\,\wt S_{1'}(u)\,R_V(u+v)\,\wt S_{2'}(v) = \wt S_{2'}(v)\,R_V(u+v)\,\wt S_{1'}(u)\,R_V(u-v) 
}
remains valid when $\wt S_{1'}(u)$ and $\wt S_{2'}(v)$ are replaced by their images under $\varphi_0$ and $\varphi_1$. Here the primed indices denote the copies of the space $V$ in the tensor product $V\ot V$. We use \eqref{RV:1} for $R_V(u-v)$ in the left-hand side, \eqref{RV:3} for $R_V(u-v)$ in the right-hand side, and \eqref{RV:2} for both $R_V(u+v)$. In such a way we obtain an equation which is of essentially the same form as the fused projected reflection equation given in Theorem 4.2 in \cite{BaRe}. The explicit form of the left-hand side of \eqref{RE:V} is $\tfrac{4u}{4u-1}\tfrac{4v}{4v-1}$ times
\spl{ \label{P52:1}
& \tfrac{1}{64} R^\circ_{12}(-1)\,R^\circ_{34}(-1)\,R^\circ_{14}(2u-2v-1)\,R^\circ_{13}(2u-2v)\,R^\circ_{24}(2u-2v)\,R^\circ_{23}(2u-2v+1) \\
& \times R^\circ_{12}(-1)\, \wt B^\circ_1(2u-1/2)\,R^\circ_{12}(4u)\, \wt B^\circ_2(2u+1/2) \\
& \times R^\circ_{12}(-1)\,R^\circ_{34}(-1)\,R^\circ_{13}(2u+2v-1)\,R^\circ_{14}(2u+2v)\,R^\circ_{23}(2u+2v)\,R^\circ_{24}(2u+2v+1) \\
& \times R^\circ_{34}(-1)\, \wt B^\circ_3(2v-1/2)\,R^\circ_{34}(4v)\, \wt B^\circ_4(2v+1/2) .
}
By \eqref{MoRa:RE} we have
\eq{ \label{P52:2}
\tfrac{1}{2} \wt B^\circ_{1'}(2u)\, R^\circ_{12}(-1) = \wt B^\circ_{1'}(2u) \qquad\text{and}\qquad \tfrac{1}{2} \wt B^\circ_{2'}(2v)\, R^\circ_{34}(-1) = \wt B^\circ_{2'}(2v) ,
}
where $\wt B^\circ_{1'}(2u)$ (resp.\ $\wt B^\circ_{2'}(2v)$) denotes the image of $\wt S_{1'}(2u)$ (resp.\ $\wt S_{2'}(2v)$) under the map $\varphi_i$. Indeed,
\eqn{
\tfrac{1}{4} \wt B^\circ_{1'}(2u)\, R^\circ_{12}(-1) &= \tfrac{4u}{4u-1}\tfrac{1}{4} R^\circ_{12}(-1)\,\wt  B^\circ_1(2u-1/2)\,R^\circ_{12}(4u)\,\wt  B^\circ_2(2u+1/2)\, R^\circ_{12}(-1) \\
&= \tfrac{4u}{4u-1} \tfrac{1}{4} \wt B^\circ_2(2u+1/2)\,R^\circ_{12}(4u)\,\wt  B^\circ_1(2u-1/2)\, R^\circ_{12}(-1)\, R^\circ_{12}(-1) = \wt B^\circ_{1'}(2u) , 
}
where in the second equality we used \eqref{MoRa:RE} and in the last equality we used the idempotence property $(\tfrac{1}{2} R^\circ_{12}(-1))^2=\tfrac{1}{2} R^\circ_{12}(-1)$ and \eqref{MoRa:RE} once again. The second relation in \eqref{P52:2} follows in a similar way. Then, using \eqref{RRR=R} and the identities above, we can shift all the projectors in \eqref{P52:1} leftwards, thus obtaining the projector $\tfrac{1}{4} R^\circ_{12}(-1)\,R^\circ_{34}(-1)$ times $\tfrac{4u}{4u-1}\tfrac{4v}{4v-1}$ and times
\eqn{
& R^\circ_{14}(2u-2v-1)\,R^\circ_{13}(2u-2v)\,R^\circ_{24}(2u-2v)\,R^\circ_{23}(2u-2v+1) \\
& \times \wt B^\circ_1(2u-1/2)\,R^\circ_{12}(4u)\, \wt B^\circ_2(2u+1/2) \\
& \times R^\circ_{13}(2u+2v-1)\,R^\circ_{14}(2u+2v)\,R^\circ_{23}(2u+2v)\,R^\circ_{24}(2u+2v+1) \el
& \times \wt B^\circ_3(2v-1/2)\,R^\circ_{34}(4v)\, \wt B^\circ_4(2v+1/2) , 
}
which is equivalent to the left-hand side of the fused reflection equation given in Theorem 4.1 in \cite{BaRe}. Now in the same way as in the proof of Theorem 4.1 in \cite{BaRe} (i.e.\ using the Yang-Baxter and the reflection equations multiple times) we obtain
\eqn{
& \wt B^\circ_3(2v-1/2)\,R^\circ_{34}(4v)\,\wt  B^\circ_4(2v+1/2) \\
& \times R^\circ_{13}(2u+2v-1)\,R^\circ_{14}(2u+2v)\,R^\circ_{23}(2u+2v)\,R^\circ_{24}(2u+2v+1) \el
& \times \wt B^\circ_1(2u-1/2)\,R^\circ_{12}(4u)\,\wt  B^\circ_2(2u+1/2) \\
& \times R^\circ_{14}(2u-2v-1) R^\circ_{24}(2u-2v) R^\circ_{13}(2u-2v) R^\circ_{23}(2u-2v+1). 
}
Then, using \eqref{RRR=R} and \eqref{P52:2}, we put the projectors into the required places and use \eqref{RV:0}, \eqref{RV:1} to obtain $\tfrac{4u}{4u-1}\tfrac{4v}{4v-1}$ times
\eqn{
& \tfrac{1}{64} R^\circ_{34}(-1)\, \wt B^\circ_3(2v-1/2)\,R^\circ_{34}(4v)\, \wt B^\circ_4(2v+1/2) \\
& \times R^\circ_{12}(-1)\,R^\circ_{34}(-1)\,R^\circ_{13}(2u+2v-1)\,R^\circ_{14}(2u+2v)\,R^\circ_{23}(2u+2v)\,R^\circ_{24}(2u+2v+1) \\
& \times R^\circ_{12}(-1)\, \wt B^\circ_1(2u-1/2)\,R^\circ_{12}(4u)\,\wt  B^\circ_2(2u+1/2) \el
& \times R^\circ_{12}(-1)\,R^\circ_{34}(-1)\,R^\circ_{23}(2u-2v-1)\,R^\circ_{24}(2u-2v)\,R^\circ_{13}(2u-2v)\,R^\circ_{14}(2u-2v+1) , 
}
which is exactly the image of the right-hand side of \eqref{RE:V}. 

We now show that the homomorphism $\varphi_q$ is surjective. The images of the matrix elements $\wt s_{ij}(u)$ of $\wt S(u)$ are found as using the method explained above \eqref{K-image}, that is  $ \wt s^\circ_{ij}(u) =\langle{ v_i \, |\, \wt S^\circ(u)\, v_j \rangle}$. In such a way, we find for instance:
\eqn{
\wt s_{00}(u) & \mapsto \tfrac{1}{2(4u-1)} \big((4 u-1)(\tl{b}^\circ_{1,-1}(2 u-1/2)\,\tl{b}^\circ_{-1,1}(2 u+1/2)+\tl{b}^\circ_{-1,1}(2 u-1/2)\,\tl{b}^\circ_{1,-1}(2 u+1/2) ) \\
& \qq\qq\qq + \tl{b}^\circ_{11}(2 u-1/2)(4 u\,\tl{b}^\circ_{-1,-1}(2 u+1/2)-\tl{b}^\circ_{11}(2 u+1/2)) \\
& \qq\qq\qq - \tl{b}^\circ_{-1,-1}(2 u-1/2)(\tl{b}^\circ_{-1,-1}(2 u+1/2)-4 u\,\tl{b}^\circ_{11}(2 u+1/2)) \big), \allowdisplaybreaks\\
\wt s_{01}(u) & \mapsto \tfrac{1}{\sqrt{2}(4u-1)} \big((\tl{b}^\circ_{-1,-1}(2 u-1/2)-4 u\,\tl{b}^\circ_{11}(2 u-1/2))\,\tl{b}^\circ_{-1,1}(2 u+1/2)\el
& \qq\qq\qq +(1-4 u)\,\tl{b}^\circ_{-1,1}(2 u-1/2)\,\tl{b}^\circ_{11}(2 u+1/2) \big), \allowdisplaybreaks\\
\wt s_{11}(u) & \mapsto \tfrac{1}{4u-1} \big((4 u-1)\,\tl{b}^\circ_{1,1}(2 u-1/2)\,\tl{b}^\circ_{11}(2 u+1/2) - \tl{b}^\circ_{1,-1}(2 u-1/2)\,\tl{b}^\circ_{-1,1}(2 u+1/2) \big) . \nn
}
We have $\tl b_{ij}^{\circ}(u) =  \delta_{ij} \epsilon_i  + \mysum_{r\ge 1} \tl b_{ij}^{\circ(r)} u^{-r}$.
Then, by taking coefficients at $u^{-r}$ on both sides of the map, we find
$$
\wt s^{(r)}_{11} \mapsto 2^{-r+1}\, \tl b^{\circ(r)}_{11} + A_{r-1}\big(\tl b^{\circ(s_1)}_{11} , \tl b^{\circ(s_2)}_{1,-1} , \tl b^{\circ(s_3)}_{-1,1}  \big)_{s=1}^{r-1}
$$
for the case $q=0$ and for any $r\geq1$; here $A_{r-1}(\ldots)$ denotes a polynomial in the  generators $\tl b^{\circ(s_1)}_{11}, \tl b^{\circ(s_2)}_{1,-1}, \tl b^{\circ(s_3)}_{-1,1}$ with $1\leq s_1,s_2,s_3\leq r-1$. This shows by induction on $r$ that each generator $\tl b^{(r)\circ}_{11}$ with $r\geq1$ belongs to the image of the homomorphism. In a similar way, by considering the images of $\wt s^{(r)}_{-1,0}$, $\wt s^{(r)}_{0,-1}$ and $\wt s^{(r)}_{-1,-1}$, we obtain equivalent properties for the generators $\tl b^{\circ(r)}_{-1,1}$, $\tl b^{\circ(r)}_{1,-1}$ and $\tl b^{\circ(r)}_{-1,-1}$, respectively, thus proving surjectivity in the $q=0$ case. The surjectivity in the $q=1$ case is obtained using the same approach and considering the images of the generators $\wt s^{(r)}_{11}$, $\wt s^{(r)}_{-1,0}$, $\wt s^{(r)}_{0,-1}$ and $\wt s^{(r)}_{-1,-1}$, respectively.
\end{proof}

\smallskip


In the following three sections, we turn our attention to those twisted Yangians associated to symmetric pairs of the form $(\mfso_4,\mfso_4^\rho)$. It will be convenient to 
 denote $\wt{u}=u-\ka/2=u-1/2$. Furthermore, we will continue to make use of the matrix $K=E_{11}-E_{-1,-1}$ throughout these sections.


\subsection{Twisted Yangians for the symmetric pair  \texorpdfstring{$(\mf{so}_4,\mf{gl}_2)$}{X(so4,gl2)}}

Our investigation will make use of the tensor products $Y(2)\otimes Y(2)$ and 
$Y^+(2)\otimes Y^-(2)$.  In order to distinguish the two copies of $Y(2)$ in the tensor product $Y(2)\otimes Y(2)$, we will follow the same convention as used in \eqref{X->Y:3}. That is, we will denote the generator series corresponding to the first copy of $Y(2)$ by $t^\circ_{ij}(u)$, and the generator series 
corresponding to the second copy by $t^\bullet_{ij}(u)$, where $i,j\in \{-1,1\}$. As usual, the formal series $t^\circ_{ij}(u)$ are arranged into the matrix $T^\circ(u)$, and the formal series $t^\bullet_{ij}(u)$ are arranged into the matrix $T^\bullet(u)$. Similarly, we denote the generator series corresponding to the orthogonal twisted Yangian in $Y^+(2)\otimes Y^-(2)$
by $s^\circ_{ij}(u)$, and those corresponding to the symplectic twisted Yangian by $s^\bullet_{ij}(u)$, where $i,j\in \{-1,1\}$. These are then arranged into the matrices $S^\circ(u)$ and $S^\bullet(u)$, respectively. 

\smallskip 

Let $\{e_{-1},e_{1}\}$ be the standard basis of $\C^2$ and set $V=\C^2\otimes \C^2$, with the ordered basis $\{v_{-2},v_{-1},v_{1},v_{2}\}$ given as follows: 
$$
 v_{-2}=e_{-1}\otimes e_{-1}, \enspace v_{-1}=e_{-1}\otimes e_{1}, \enspace v_1=e_1\otimes e_{-1},\enspace v_2=-e_{1}\otimes e_1.
$$
Then we may consider $T^\circ(u)\,T^\bullet(u)$ as an element of $\End V \otimes \left( Y(2)\otimes Y(2)\right)[[u^{-1}]]$ and, in particular, it is the image 
of the $T$-matrix $T(u)$ of $X(\mfso_4)$ under the embedding $\psi_3$ given in \eqref{X->Y:3}.

\smallskip

Restricting $\psi_3$ to the subalgebra $X(\mfso_4,\mfgl_2)^{tw}$ of $X(\mfso_4)$ leads to the following result:
  
\begin{prop}\label{DIII:Prop.embedding}
The assignment
\begin{equation}
 S(u)\mapsto S^\circ(u-1/2)\,K_1 \,S^\bullet(u-1/2) \label{DIII:embedding.1}
\end{equation}
defines an embedding $\varphi_0:X(\mathfrak{so}_4,\mathfrak{gl}_2)^{tw}\hookrightarrow Y^+(2)\otimes Y^-(2)$. 
\end{prop}   

\begin{proof}
Since $X(\mfso_4,\mfgl_2)^{tw}$ is a subalgebra of $X(\mfso_4)$, it is already clear that if we set $\varphi_0=\psi_3|_{X(\mfso_4,\mfgl_2)^{tw}}$,
 then we get an embedding: $\varphi_0:X(\mfso_4,\mfgl_2)^{tw}\hookrightarrow Y(2)\otimes Y(2)$. Hence, we have to see that the image of this embedding is contained in  $Y^+(2)\otimes Y^-(2)$, where here $Y^+(2)$ and $Y^-(2)$ are both identified as subalgebras of $Y(2)$ via the embedding \eqref{S=TGT}, and in addition that it is given by the assignment \eqref{DIII:embedding.1}.
 
\smallskip 
 
Viewing $X(\mfso_4,\mfgl_2)^{tw}$ as a subalgebra of $X(\mfso_4)$, we have 
\begin{equation*}
 S(u)=T(u-1/2)\,\mcG\, T(-u+1/2)^{t_+} =T(\wt{u})\,\mcG\, T(-\wt{u})^{t_+}. 
\end{equation*}
It can be checked directly from the explicit expressions in \cite[Proposition 4.8]{AMR} that $\psi_3(T(u)^{t_+}) = T^{\circ}(u)^{t_-} T^{\bullet}(u)^{t_-}$. Moreover, $\mathcal{G}=K_1$ when $\mcG$ is viewed as an element of $\End\, \C^2 \otimes \End\, \C^2$ and we have that $K_1\, T^{\circ}(-\wt{u})^{t_-} K_1 = T^{\circ}(-\wt{u})^{t_+}$. Therefore, 
\begin{equation}
  \psi_3(S(u))=\psi_3(T(\wt{u}))\, K_1\, \psi_3(T(-\wt{u})^{t_+})=T^\circ(\wt{u})\,T^\bullet(\wt{u})\, K_1\,  T^{\circ}(-\wt{u})^{t_-}\,T^{\bullet}(-\wt{u})^{t_-} = S^\circ(\wt{u})\,K_1\, S^\bullet(\wt{u}). \label{DIII.embedding.2}
\end{equation}

This proves that $\Ima\,\varphi_0\subset Y^+(2)\otimes Y^-(2)$ and that $\varphi_0$ is given by the assignment \eqref{DIII:embedding.1}. 
\end{proof}

To compute the explicit images of each generating series $s_{ij}(u)$ under $\varphi_0$, we only have to use the formulas given in \cite[Proposition 4.8]{AMR} and replace all series $t_{ab}(u)$ by $s_{ab}^\circ(\wt{u})$ and all series $t^\bullet_{ab}(u)$ by $s_{ab}^\bullet(\wt{u})$, taking into account the sign changes coming from the appearance of $K_1$ in \eqref{DIII:embedding.1}.

Recall that the projection map $Y(2)\to SY(2)$ given by $t^\circ_{ij}(u)\mapsto \tau^\circ_{ij}(u)$ defines a surjective homomorphism of algebras. We denote 
 this map by $\mathrm{Pr}$. In \cite[Corollary 4.10]{AMR}, it was shown that the maps $\chi^{(1)}=\mathrm{Pr}_1\circ \psi_3$, and $\chi^{(2)}=\mathrm{Pr}_2\circ \psi_3$
 provide isomorphisms
$$
X(\mathfrak{so}_4)\cong SY(2)\otimes Y(2) \qu\text{and}\qu X(\mathfrak{so}_4)\cong  Y(2)\otimes SY(2), \qu\text{respectively.}
$$

\smallskip 
 
The restriction of the projection $\mathrm{Pr}$ to $Y^{\pm}(2)$ coincides with the natural projection of $Y^{\pm}(2)$ onto $SY^{\pm}(2)$. Thus, $\mathrm{Pr}(s^\circ_{ij}(u))$ is equal to $\si^\circ_{ij}(u)$ and the matrix $\mathrm{Pr}(S^\circ(u))$ is equal to $\Si^\circ(u)$, as in Definition \ref{D:Y(g,G):tw}.
We have the following corollary of Proposition \ref{DIII:Prop.embedding}:

\begin{crl}\label{DIII:Cor.isomorphism}
The compositions $\chi^{(1)}_0=\mathrm{Pr}_1\circ \varphi_0$ and $\chi^{(2)}_0=\mathrm{Pr}_2\circ \varphi_0$ define isomorphisms 
\begin{equation}
 X(\mathfrak{so}_4,\mathfrak{gl}_2)^{tw}\cong SY^+(2)\otimes Y^-(2) \quad \mathrm{ and } \quad X(\mathfrak{so}_4,\mathfrak{gl}_2)^{tw}\cong Y^+(2)\otimes SY^-(2), \text{ respectively.}
\end{equation}
\end{crl}

\begin{proof}
We show that $\mathrm{Pr}_1\circ \varphi_0$  is an isomorphism. The argument for $\mathrm{Pr}_2\circ \varphi_0$ is similar. Note first that $\mathrm{Pr}_1\circ \varphi_0=\chi^{(1)}|_{X(\mathfrak{so}_4,\mathfrak{gl}_2)^{tw}}$. Therefore, since $\chi^{(1)}$ is an isomorphism, $\mathrm{Pr}_1\circ \varphi_0$ is injective. Moreover, it is clear that the image of $\mathrm{Pr}_1\circ \varphi_0$ is contained in $SY^+(2)\otimes Y^-(2)$.

To show surjectivity, we use an argument similar to that used in the proof of Corollary 4.10 in \cite{AMR},
where the role of the quantum determinant ${\rm qdet}\, T^\circ (u)$ is played by the Sklyanin determinant ${\rm sdet}\,S^\circ(u)$. We first show that the coefficients of the series $s^\bullet_{11}(u)$ belong to the image of $\chi^{(1)}_0$. By \eqref{DIII:embedding.1}, the definition of $\chi^{(1)}_0$  and  the relation \eqref{Ols:sdet}, we have
\begin{align*}
 \chi^{(1)}_0&\left(s_{11}(-\wt{u})\,s_{22}(\wt{u})-s_{1,-2}(-\wt{u})\,s_{-1,2}(\wt{u}) \right)\\
                   &=\si^\circ_{11}(-u)\,s^\bullet_{-1,-1}(-u)\,\si^\circ_{11}(u-1)\,s^\bullet_{11}(u-1)-\si^\circ_{1,-1}(-u)\,s^\bullet_{-1,-1}(-u)\,\si^\circ_{-1,1}(u-1)\,s^\bullet_{11}(u-1)\\
                   &=\left(\si^\circ_{11}(-u)\,\si^\circ_{11}(u-1)-\si^\circ_{1,-1}(-u)\,\si^\circ_{-1,1}(u-1)\right)s^\bullet_{-1,-1}(-u)\,s^\bullet_{11}(u-1)\\
                   &=s^\bullet_{-1,-1}(-u)\,s^\bullet_{11}(u-1),
\end{align*}
where the last equality follows from the fact that ${\rm sdet}\,\Si^\circ(u)=1$. By the symmetry relation \eqref{Ols:symm} we have $s^\bullet_{-1,-1}(-u)=s^\bullet_{11}(u)-\frac{1}{2}(s^\bullet_{11}(u)-s^\bullet_{11}(-u))\,u^{-1}$. Thus the coefficients of the series $s^\bullet_{-1,-1}(-u)\,s^\bullet_{11}(u-1)$, which equals
\begin{equation}
\left(s^\bullet_{11}(u)-\tfrac{1}{2u}\left(s^\bullet_{11}(u)-s^\bullet_{11}(-u)\right) \right)s^\bullet_{11}(u-1), \label{DIII:isomorphism.1}
\end{equation}
belongs to $\Ima\, \chi_0^{(1)}$. This implies that all the coefficients of $s^\bullet_{11}(u)$ belong to $\Ima\,\chi_0^{(1)}$. To see this, expand the series \eqref{DIII:isomorphism.1} in $Y^-(2)[[u^{-1}]]$ and argue by induction.

Next, since $s^{\bullet(0)}_{11}=1$ the series $s^\bullet_{11}(u)$ has an inverse in $Y^-(2)[[u^{-1}]]$.  In particular, since there is a series $\mathsf{z}(u)\in X(\mfso_4,\mfgl_2)^{tw}[[u^{-1}]]$ satisfying $\chi^{(1)}_0(\mathsf{z}(u))=s^\bullet_{11}(u)$, writing 
$
\mathsf{z}(u)=\sum_{r\ge 0} \mathsf{z}^{(r)}u^{-r},
$
we must also have $\mathsf{z}^{(0)}=1$. Hence, $\mathsf{z}(u)$ is also invertible with inverse $\mathsf{z}^{-1}(u)$ which satisfies $\chi^{(1)}_0\left(\mathsf{z}^{-1}(u)\right)=s^\bullet_{11}(u)^{-1}$. 
This implies that $\si^\circ_{ij}(u)\in \Ima\,\chi^{(1)}_0[[u^{-1}]]$ for all $i,j\in \{-1,1\}$. Indeed, 
the formula \eqref{DIII:embedding.1} and the definition of $\chi_0^{(1)}$ imply the following equalities:
\spl{
\si^\circ_{-1,-1}(u)&= \chi^{(1)}_0\left(-s_{-1,-1}(u+1/2)\,\mathsf{z}(u)^{-1}\right), \qq  \si^\circ_{11}(u)= \chi^{(1)}_0\left(s_{22}(u+1/2)\,\mathsf{z}(u)^{-1} \right) \\
\si^\circ_{-1,1}(u)&= \chi^{(1)}_0\left(-s_{-1,2}(u+1/2)\,\mathsf{z}(u)^{-1}\right), \qq \si^\circ_{1,-1}(u)=\chi^{(1)}_0\left(s_{2,-1}(u+1/2)\,\mathsf{z}(u)^{-1}\right).
}
Similarly, since $\si^{\circ(0)}_{11}=1$, the series $\si^\circ_{11}(u)$ is invertible. As there exists a series $\mathsf{y}(u)$ in $X(\mfso_4,\mfgl_2)^{tw}[[u^{-1}]]$ satisfying $\chi^{(1)}_0(\mathsf{y}(u))=\si^\circ_{11}(u)$, using an argument similar to that just used with $\mathsf{z}(u)$, we obtain $s^\bullet_{ij}(u)\in \Ima\,\chi^{(1)}_0$ for all $i,j\in \{\pm1\}$ (this has already been established for the case $i=j=1$). This completes the proof that $\chi^{(1)}_0$ is surjective, and thus that $\chi^{(1)}_0$ is an isomorphism. 
\end{proof}

\begin{crl} \label{DIII:Cor.Yangian}
The isomorphism $\chi_0^{(0)}$ descends to the quotients $Y(\mfso_4,\mfgl_2)^{tw}$ and $ SY^+(2)\ot SY^-(2)$, thus inducing an isomorphism 
\begin{equation}
Y(\mfso_4,\mfgl_2)^{tw}\cong  SY^+(2)\otimes SY^-(2). \label{DIII:Yangian}
\end{equation}
\end{crl}

\begin{proof}
Recall that in $X(\mfg,\mcG)^{tw}$ we have the relation $S(u)S(-u)=w(u)\cdot I$, where $w(u)$ is an even formal power series whose coefficients generate the center of $X(\mfg,\mcG)^{tw}$. In particular, this implies that
$$
\chi^{(1)}_0(S(u)\,S(-u))=\chi^{(1)}_0(w(u))\cdot I \otimes I,
$$ 
where $\chi^{(1)}_0(w(u))$ is a central series in $SY^+(2)\otimes Y^-(2)$.  Conversely, the formulas of Proposition \ref{DIII:Prop.embedding} give:
\begin{align*}
\chi^{(1)}_0(S(u)\,S(-u))&=\Si^\circ(\wt{u})\,K_1\,S^\bullet(\wt{u})\,\Si^\circ(-\wt{u}-1)\,K_1\,S^\bullet(-\wt{u}-1)\\
                         &=\left(\Si^\circ(\wt{u})\,K_1\, \Si^\circ(-\wt{u}-1)\,K_1\right)\left(S^\bullet(\wt{u})\,S^\bullet(-\wt{u}-1) \right).
\end{align*}
Therefore both $\Si^\circ(\wt{u})\,K_1\, \Si^\circ(-\wt{u}-1)K_1$ and $S^\bullet(\wt{u})\,S^\bullet(-\wt{u}-1)$ are multiples of the identity $I$. However, one easily checks that, for $i=\pm1$, the coefficient of $E_{ii}$ in $\Si^\circ(\wt{u})\,K_1\, \Si^\circ(-\wt{u}-1)\,K_1$ is given by 
$$
\mathrm{sgn}(i)\,\si^\circ_{i1}(\wt{u})\,\si^\circ_{1i}(-\wt{u}-1)-\mathrm{sgn}(i)\,\si^\circ_{i,-1}(\wt{u})\,\si^\circ_{-1,i}(-\wt{u}-1)={\rm sdet}\, \Si^\circ(\alpha_i(u))=1,
$$
where ${\alpha_i(u)}=\delta_{i,-1}-\mathrm{sgn}(i)\,\wt{u}$. Here we have used the formulas in \eqref{Ols:sdet} for the Sklyanin determinant, as well as the fact that $SY^+(2)$ is equal to the quotient $Y^+(2)/({\rm sdet}\,S^\circ(u)-1)$. 

Similarly, an easy computation using relation \eqref{Ols:sdet} shows that the diagonal entries of $S^\bullet(\wt{u})S^\bullet(-\wt{u}-1)$ are equal to 
$$
s^\bullet_{i1}(\wt{u})\,s^\bullet_{1i}(-\wt{u}-1)+s^\bullet_{i,-1}(\wt{u})\,s^\bullet_{-1,i}(-\wt{u}-1)=\frac{2{\alpha_i(u)}-1}{2{\alpha_i(u)}+1}\, {\rm sdet}\,S^\bullet({\alpha_i(u)}).
$$
However, by \eqref{Ols:qdet->sdet} the right hand side of the above equation is equal to $\gamma_2(-\wt{u})^{-1}{\rm sdet}\,S^\bullet(-\wt{u}) $, where $\gamma_2(u)=(2u+1)(2u-1)^{-1}$, as in \eqref{gamma(u)}). Therefore we can write 
\begin{equation*}
\chi^{(1)}_0(w(u))=1\otimes \gamma_2(-\wt{u})^{-1}\,{\rm sdet}\,S^\bullet(-\wt{u}),
\end{equation*}
and so $\chi^{(1)}_0$ sends the ideal generated by the nonconstant coefficients of $w(u)$ to the ideal $SY^+(2)\otimes J$, where $J$ is the ideal in $Y^-(2)$ generated by the non-constant coefficients of $\gamma_2(-\wt{u})^{-1}\,{\rm sdet}\,S^\bullet(u)$. 
By \eqref{quotAIAII}, $SY^-(S)\cong Y^-(2)/J$, and as consequence we have the following sequence of induced isomorphisms: 
\begin{equation*}
 Y(\mfso_4,\mfgl_2)^{tw}\cong (SY^+(2)\otimes Y^-(2))/(SY^+(2)\otimes J)\cong SY^+(2)\otimes \left( Y^-(2)/J\right)\cong SY^+(2)\otimes SY^-(2).
\end{equation*}
This establishes the isomorphism \eqref{DIII:Yangian}, completing the proof of the corollary.

We can also obtain the isomorphism $Y(\mfso_4,\mfgl_2)^{tw}\cong  SY^+(2)\otimes SY^-(2)$ by viewing $Y(\mfso_4,\mfgl_2)^{tw}$ as a subalgebra of $X(\mfso_4,\mfgl_2)^{tw}$ and restricting the
isomorphism $\chi^{(1)}_0$. Indeed, recall that  $Y(\mfso_4,\mfgl_2)^{tw}$ can be identified as the subalgebra of $X(\mfso_4,\mfgl_2)^{tw}$ stable under all the automorphisms $\nu_g$ of the form
$S(u)\mapsto g(u-1/2)S(u)$, where $g(u)$ is an even formal power series in $1+u^{-2}\C[[u^{-2}]]$. 
Similarly $SY^-(2)$ is isomorphic to the subalgebra of $Y^-(2)$ stable under all automorphisms $\nu^\bullet_h$ of the form 
$S^\bullet(u)\mapsto h(u)S^\bullet(u)$, where again $h(u)$ belongs to $1+u^{-2}\C[[u^{-2}]]$. Therefore, $SY^+(2)\otimes SY^-(2)$ is isomorphic to the subalgebra 
of $SY^+(2)\otimes Y^-(2)$ stable under all automorphisms $1\otimes \nu^\bullet_h$. Moreover, given a series $g(u)\in 1+u^{-2}\C[[u^{-2}]]$, it follows immediately from the definition of $\chi^{(1)}_0$ that 
$(1\otimes \nu^\bullet_g)\circ \chi^{(1)}_0=\chi^{(1)}_0\circ\nu_g$. This implies that the image of $Y(\mfso_4,\mfgl_2)^{tw}$ under the isomorphism $\chi^{(1)}_0$ is exactly the subalgebra $SY^+(2)\otimes SY^-(2)$
of $SY^+(2)\otimes Y^-(2)$. 
\end{proof}

Note that the previous corollary can be proven in the same way with $\chi^{(1)}_0$ replaced by $\chi^{(2)}_0$. This will also be true for the corresponding corollaries \ref{D0:Cor.Yangian} and \ref{D1:Cor.Yangian}  in the following two subsections. 
 
\begin{crl}\label{Yso4gl2B}
The isomorphisms obtained in Corollary \ref{DIII:Cor.isomorphism} and Corollary \ref{DIII:Cor.Yangian} can be written in terms of the reflection algebras $\mathcal{B}(2,q)$ and $\mathcal{SB}(2,q)$ studied
in \cite{MR}. Specifically, we obtain isomorphisms
\begin{equation*}
 X(\mfso_4,\mfgl_2)^{tw}\cong \mathcal{SB}(2,1)\otimes \mcB(2,0) \quad \mathrm{ and } \quad  X(\mfso_4,\mfgl_2)^{tw}\cong \mcB(2,1)\otimes \mathcal{SB}(2,0),
\end{equation*}
while
\begin{equation*}
 Y(\mfso_4,\mfgl_2)^{tw}\cong \mathcal{SB}(2,1)\otimes \mathcal{SB}(2,0).
\end{equation*}
\end{crl}

\begin{proof} 
This is a corollary of the isomorphisms $SY^\pm(2)\cong \mathcal{SB}(2,q)$ which are given explicitly in \eqref{SYSB0} and \eqref{SYSB1}, together with the fact that, due to \eqref{TY=TZ*TSY}, these isomorphisms 
can be extended to obtain isomorphisms $Y^\pm(2)\cong\mcB(2,q)$.
\end{proof}


\subsection{Twisted Yangians for the symmetric pair \texorpdfstring{$(\mf{so}_4,\mf{so}_4)$}{X(so4,so4)}}

We will establish results analogous to those obtained for $X(\mfso_4,\mfgl_2)^{tw}$, this time 
replacing the tensor product $Y^+(2)\otimes Y^-(2)$ with $Y^-(2)\otimes Y^-(2)$. We keep all the same notation, with the slight modification that the notation $s^\circ_{ij}(u)$ will be
reserved for the series in the $(i,j)^{th}$ entry of the $S$-matrix $S^\circ(u)$ corresponding to the first copy of $Y^-(2)$, as opposed to $Y^+(2)$.

\begin{prop}\label{D0:Prop.embedding}
The assignment
\begin{equation}
S(u)\mapsto S^\circ(u-1/2)\,S^\bullet(u-1/2) \label{D0:embedding.1}
\end{equation}
defines an embedding $\varphi_1:X(\mfso_4,\mfso_4)^{tw}\hookrightarrow Y^-(2)\otimes Y^-(2)$. 
\end{prop}   

\begin{proof}
The argument is very similar to that used in the proof of Proposition \ref{DIII:Prop.embedding}. We set $\varphi_1=\psi_3|_{X(\mfso_4,\mfso_4)^{tw}}$, an embedding, and show that $\varphi_1$ is actually given by the assignment \eqref{D0:embedding.1}. Viewing $X(\mfso_4,\mfso_4)^{tw}$ as a subalgebra of the extended Yangian $X(\mfso_4)$, we have 
$$
S(u)=T(u-1/2)\,T(-u+1/2)^{t_+} = T(\wt{u})\,T(-\wt{u})^{t_+}.
$$
Therefore, applying $\psi_3$ to $S(u)$ we obtain: 
\begin{equation}
\psi_3(S(u))=\psi_3(T(\wt{u}))\,\psi_3(T(-\wt{u})^{t_+}) = T^\circ(\wt{u})\,T^\bullet(\wt{u})\, T^{\circ}(-\wt{u})^{t_-}T^{\bullet}(-\wt{u})^{t_-} = S^\circ(\wt{u})\,S^\bullet(\wt{u}). \label{D0:embedding.2}
\end{equation}
\end{proof}

The image of all series $t_{ij}(u)$ under $\psi_3$ are given in \cite[Proposition 4.8]{AMR}. To obtain the image of $s_{ij}(u)$ under $\varphi_1$, one need only observe where $\psi_3$ maps $t_{ij}(u)$ and then replace all series $t^\circ_{ab}(u)$ by $s^\circ_{ab}(\wt{u})$, and all series $t^\bullet_{ab}(u)$ by $s^\bullet_{ab}(\wt{u})$.

The following corollary gives isomorphisms for $X(\mfso_4,\mfso_4)^{tw}$ analogous to those obtain for $X(\mfso_4,\mfgl_2)^{tw}$ in Proposition \ref{DIII:Cor.isomorphism}.

\begin{crl}\label{D0:Cor.isomorphism}
The compositions $\chi^{(1)}_1=\mathrm{Pr}_1\circ \varphi_1$ and $\chi^{(2)}_1=\mathrm{Pr}_2\circ \varphi_1$ define isomorphisms 
\begin{equation}
X(\mfso_4,\mfso_4)^{tw}\cong SY^-(2)\otimes Y^-(2) \quad \mathrm{ and } \quad X(\mfso_4,\mfso_4)^{tw}\cong Y^-(2)\otimes SY^-(2),
\end{equation}
respectively. 
\end{crl}

\begin{proof}
As in the proof of Corollary \ref{DIII:Cor.isomorphism}, we consider only $\chi^{(1)}_1$. By \eqref{D0:embedding.1} and the definition of $\chi^{(1)}_1$, we have
\begin{align*}
&\chi^{(1)}_1\left(s_{11}(-\wt{u})\,s_{22}(\wt{u})-s_{1,-2}(-\wt{u})\,s_{-1,2}(\wt{u})\right)\\
                   &=\left(\si^\circ_{11}(-u)\,s^\bullet_{-1,-1}(-u)\,\si^\circ_{11}(u-1)\,s^\bullet_{11}(u-1)+\si^\circ_{1,-1}(-u)\,s^\bullet_{-1,-1}(-u)\,\si^\circ_{-1,1}(u-1)\,s^\bullet_{11}(u-1)\right)\\
                   &=\left(\si^\circ_{11}(-u)\,\si^\circ_{11}(u-1)+\si^\circ_{1,-1}(-u)\,\si^\circ_{-1,1}(u-1)\right) s^\bullet_{-1,-1}(-u)\,s^\bullet_{11}(u-1)\\
                   &=s^\bullet_{-1,-1}(-u)\,s^\bullet_{11}(u-1),
\end{align*}
where we have obtained the last equality as consequence of \eqref{Ols:sdet} and the relation $\frac{2u-1}{2u+1}\,{\rm sdet}\,\Si^\circ(u)=1$ in $SY^-(2)$. This implies that the coefficients of the series $s^\bullet_{-1,-1}(-u)\,s^\bullet_{11}(u-1)$ are contained in the image of $\chi^{(1)}_1$. The rest of the proof now proceeds identically to the proof of Corollary \ref{DIII:Cor.isomorphism}. 
\end{proof}

\begin{crl}\label{D0:Cor.Yangian} 
The isomorphism $\chi^{(1)}_1$ descends to the quotients $Y(\mfso_4,\mfso_4)^{tw}$ and $ SY^-(2)\otimes SY^-(2)$, thus inducing an isomorphism 
\begin{equation}
Y(\mfso_4,\mfso_4)^{tw}\cong  SY^-(2)\ot SY^-(2). \label{D0:Yangian}
\end{equation}
\end{crl}

\begin{proof}
This corollary is proved the same way as Corollary \ref{DIII:Cor.Yangian}. One first shows that 
\begin{equation}  
\chi^{(1)}_1(w(u))= 1\ot \Big(\frac{u}{u-1}\Big)\, {\rm sdet}\,S^\bullet(-\wt{u})=1\ot\ga_2(-\wt{u})^{-1}\,{\rm sdet}\,S^\bullet(-\wt{u}),
\end{equation}
which implies that the ideal $(w(u)-1)$ in $X(\mfso_4,\mfso_4)^{tw}$ is sent to $SY^+(2)\otimes J$, where $J$ is the ideal $\left({\rm sdet}\,S^\bullet(u)-\gamma_2(u)\right)$ in $Y^-(2)$. This correspondence induces the
isomorphism \eqref{D0:Yangian}.
\end{proof}

We have an analogue of Corollary \ref{Yso4gl2B}.

\begin{crl}
Since $Y^-(2)\cong\mcB(2,0)$ and $SY^-(2)\cong \mathcal{SB}(2,0)$, the isomorphisms introduced in Corollary \ref{D0:Cor.isomorphism} and Corollary \ref{D0:Cor.Yangian} can be expressed in terms of the reflection algebras $\mcB(2,0)$ and $\mathcal{SB}(2,0)$ as follows (see \eqref{SYSB0}):
\begin{equation*}
 X(\mfso_4,\mfso_4)^{tw}\cong \mathcal{SB}(2,0)\ot \mcB(2,0) \quad \mathrm{ and } \quad  X(\mfso_4,\mfso_4)^{tw}\cong \mcB(2,0)\ot \mathcal{SB}(2,0),
\end{equation*}
while 
\begin{equation*}
Y(\mfso_4,\mfso_4)^{tw}\cong \mathcal{SB}(2,0)\otimes \mathcal{SB}(2,0).
\end{equation*}
\end{crl}


\subsection{Twisted Yangians for the symmetric pairs  \texorpdfstring{$(\mfso_4,\mfso_2\op\mfso_2)$}{(so4,so_2+so_2)}}

In this section we aim to establish similar results to the last two sections for the extended twisted Yangian $X(\mfso_4,\mfso_2\oplus\mfso_2)^{tw}$ of type DI. Recall that $\mcG=\sum_{i=1}^2(E_{i,-i}+E_{-i,i})$. Define the matrix $A\in SL_4(\C)$ by 
$$
A=\frac{1}{2}\text{\small$\begin{pmatrix}
        i & -1 & -1 & i\\
        i & 1 & -1 & -i\\
        i & -1 & 1 & -i\\
       -i & -1 & -1 & -i
       \end{pmatrix}$}
\qu\text{giving}\qu
A\mcG A^t=\text{\small$\begin{pmatrix}
              1 & 0 & 0 & 0\\
              0 & -1 & 0 & 0\\
              0 & 0 & -1 & 0\\
              0 & 0 & 0 & 1
            \end{pmatrix}$}.
$$
Notice that $A^t=A^{-1}$, so $X(\mfso_4,\mcG)^{tw}\cong X(\mfso_4,A\mcG A^{t})^{tw}$. Accordingly, for the remainder of this section, we shall replace the matrix $\mcG$ with the matrix $\mcG^\prime$ given by $\mcG^\prime=A\mcG A^t$. 

\smallskip 
 
For the remainder of this section, we will be working with the tensor product $Y^+(2)\otimes Y^+(2)$. We keep the same notation for tensor products of generators used in the last two sections. 
       
\begin{prop}\label{DI:Prop.embedding}
The assignment
\eq{
S(u)\mapsto S^\circ(u-1/2)\,K_1\,S^\bullet(u-1/2)\,K_2 \label{DI:embedding.1}
}
defines an embedding $\varphi_2:X(\mfso_4,\mfso_2\oplus \mfso_2)^{tw}\hookrightarrow Y^+(2)\otimes Y^+(2)$. 
\end{prop}   
       
\begin{proof}
The argument is similar to that used in Propositions \ref{DIII:Prop.embedding} and \ref{D0:Prop.embedding}, so we omit most details. Set $\varphi_2=\psi_3|_{X(\mfso_4,\mfso_2\oplus\mfso_2)^{tw}}$. Since we already know this defines an embedding, to complete the proof of the proposition, we only have to show that it is given by the assignment \eqref{DI:embedding.1}. We have 
$$
S(u)=T(u-1/2)\,\mcG^\prime\, T(-u+1/2)^{t_+}.
$$
Therefore, 
\begin{align*}
\psi_3(S(u)) = {} & \psi_3(T(\wt{u}))\,\mcG^\prime\,\psi_3(T(-\wt{u})^{t_+}) = T^\circ(\wt{u})\, T^\bullet(\wt{u})\,\mcG^\prime \, T^{\circ}(-\wt{u})^{t_-} T^{\bullet}(-\wt{u})^{t_-} \\
= {} & T^\circ(\wt{u})\, K_1\, T^{\circ}(-\wt{u})^{t_-} T^\bullet(\wt{u}) \,K_2\, T^{\bullet}(-\wt{u})^{t_-} \qu\text{since } \mcG^\prime= K_1 K_2; \\
= {} & T^\circ(\wt{u})\, T^{\circ}(-\wt{u})^{t_+} K_1\, T^\bullet(\wt{u})\,T^{\bullet}(-\wt{u})^{t_+} K_2 \\
= {} &  S^\circ(u-1/2)\,K_1\,S^\bullet(u-1/2)\,K_2 .
\end{align*}
\end{proof}

Explicit formulas for the images of the generators $s_{ij}(u)$ can then be obtain by replacing $t_{ab}(u)$ by $s_{ab}^{\circ}(\wt{u})$ and $t^{\prime}_{ab}(u)$ by $s_{ab}^{\bullet}(\wt{u})$ in \cite[Proposition 4.8]{AMR} and inserting signs appropriately to take $K_1$ and $K_2$ into account.


\begin{crl}\label{D1:Cor.isomorphism}
The compositions $\chi^{(1)}_2=\mathrm{Pr}_1\circ \varphi_2$ and $\chi^{(2)}_2=\mathrm{Pr}_2\circ\varphi_2$ define isomorphisms 
\begin{equation}
X(\mfso_4,\mfso_2\oplus \mfso_2)^{tw}\cong SY^+(2)\otimes Y^+(2) \quad \mathrm{ and } \quad X(\mfso_4,\mfso_2\oplus \mfso_2)^{tw}\cong Y^+(2)\otimes SY^+(2),
\end{equation}
respectively. 
\end{crl}

\begin{proof}
As in the proof of Corollaries \ref{DIII:Cor.isomorphism} and \ref{D0:Cor.isomorphism}, we consider only $\chi^{(1)}_2$. By \eqref{DI:embedding.1} and definition of $\chi^{(1)}_2$, we have
\begin{align*}
\chi^{(1)}_2 & \left(-s_{11}(-\wt{u})\,s_{22}(\wt{u})+s_{1,-2}(-\wt{u})\,s_{-1,2}(\wt{u})\right)\\
                   &=\left(\si^\circ_{11}(-u)\,s^\bullet_{-1,-1}(-u)\,\si^\circ_{11}(u-1)s^\bullet_{11}(u-1)-\si^\circ_{1,-1}(-u)\,s^\bullet_{-1,-1}(-u)\,\si^\circ_{-1,1}(u-1)\,s^\bullet_{11}(u-1)\right)\\
                   &=\left(\si^\circ_{11}(-u)\,\si^\circ_{11}(u-1)-\si^\circ_{1,-1}(-u)\,\si^\circ_{-1,1}(u-1)\right)s^\bullet_{-1,-1}(-u)\,s^\bullet_{11}(u-1)\\
                   &=s^\bullet_{-1,-1}(-u)\,s^\bullet_{11}(u-1),
\end{align*}
where we have obtained the last equality as consequence of relation \eqref{Ols:sdet} and the fact that ${\rm sdet}\,\Si^\circ(u)=1$ in $SY^+(2)$. This implies that the coefficients of the series $s^\bullet_{-1,-1}(-u)\,s^\bullet_{11}(u-1)$ are contained in the image of $\chi^{(1)}_2$. The rest of the proof now proceeds identically to the proof of Corollary \ref{DIII:Cor.isomorphism}. 
\end{proof}

\begin{crl}\label{D1:Cor.Yangian}
The isomorphism $\chi^{(1)}_2$ descends to the quotients $Y(\mfso_4,\mfso_2\oplus \mfso_2)^{tw}$ and $SY^+(2)\otimes SY^+(2)$, thus inducing an isomorphism 
\begin{equation}
Y(\mfso_4,\mfso_2\oplus \mfso_2)^{tw}\cong  SY^+(2)\otimes SY^+(2). \label{D1:Yangian}
\end{equation}
\end{crl}

\begin{proof}
This corollary is proved the same way as corollaries \ref{DIII:Cor.Yangian} and \ref{D0:Cor.Yangian}. One first shows that 
\begin{equation}  
\chi^{(1)}_2(w(u))=1\otimes {\rm sdet}\,S^\bullet(-\wt{u}),
\end{equation}
which implies that the ideal $(w(u)-1)$ in $X(\mfso_4,\mfso_2\oplus\mfso_2)^{tw}$ is sent to $SY^+(2)\otimes J$, where $J$ is the ideal $\left({\rm sdet}\,S^\bullet(u)-1\right)$ in $Y^+(2)$. This correspondence induces the isomorphism \eqref{D1:Yangian}.
\end{proof}

\begin{crl}
Since $Y^+(2)\cong\mcB(2,1)$ and $SY^+(2)\cong \mathcal{SB}(2,1)$, the isomorphisms introduced in Corollary \ref{D1:Cor.isomorphism} and Corollary \ref{D1:Cor.Yangian} can be expressed
in terms of the reflection algebras $\mcB(2,1)$ and $\mathcal{SB}(2,1)$ as follows (see \eqref{SYSB1}):
\begin{equation*}
X(\mfso_4,\mfso_2\oplus\mfso_2)^{tw}\cong \mathcal{SB}(2,1)\otimes \mcB(2,1) \quad \mathrm{ and } \quad  X(\mfso_4,\mfso_2\oplus\mfso_2)^{tw}\cong \mcB(2,1)\otimes \mathcal{SB}(2,1),
\end{equation*}
and  
\begin{equation*}
 Y(\mfso_4,\mfso_2\oplus\mfso_2)^{tw}\cong \mathcal{SB}(2,1)\otimes \mathcal{SB}(2,1).
\end{equation*}
\end{crl}


\section{Isomorphisms with twisted Yangians in the Drinfeld's original presentation}

In the previous sections, we obtained isomorphisms for twisted Yangians of small rank in the $RTT$-presentation. Here we will determine isomorphisms with twisted Yangians of $\mfsl_2$ defined in Drinfeld's original presentation \cite{Dr1,Dr2,BeRe}.  We will use the calligraphic  letter $\mcY$ to denote Yangians in this presentation.


\subsection{Yangians and twisted Yangians in Drinfeld's original presentation}

The next definition is due to Drinfeld \cite{Dr1}.

\begin{defn}\label{JdefY2}
The Yangian $\mcY_{2}$ in Drinfeld's original presentation is the unital associative $\C$-algebra generated by the elements $h, e, f$ and $J(h), J(e), J(f)$ satisfying
\eqa{ 
& [h,e]=2e, \quad [h,f]=-2f, \quad [e,f]=h, \label{Y2:Lie} \\
& [h,J(e)]=[J(h),e]=2J(e), \quad [h,J(f)]=[J(h),f]=-2J(f), \quad [e,J(f)]=[J(e),f]=J(h), \label{Y2:LJ} \\
& \big[J(h),[J(e),J(f)]\big] = \big(J(f)\,e-fJ(e)\big)h. \label{Y2:JJJ} 
}
The Hopf algebra structure on $\mcY_{2}$ is given by
\eqa{
&\Delta(x)=x\ot1+1\ot x, \qq
\Delta(J(x))=J(x)\ot1+1\ot J(x) + \frac{1}{2}\,[x\ot1,\Omega] , \el
&\qq S(x)=-x , \qq S(J(x))=-J(x)+ x ,\qq \epsilon(x)=\epsilon(J(x))=0, \nn
}
for all $x\in\{h,e,f\}$; here $\Omega=e\ot f + f\ot e + \frac{1}{2}h\ot h$ is the two-site quadratic Casimir element. The Yangian $\mcY_{2}$ becomes a filtered algebra if we set $\deg(x)=0$ and $\deg(J(x))=1 \; \forall \, x\in\{h,e,f\} $.
\end{defn}

The next two definitions are due to S. Belliard and N. Crampe \cite{BeCr} (also see \cite[Section 5]{BeRe}).

\begin{defn} 
The orthogonal twisted Yangian $\mcY^+_{2}$ in Drinfeld's original presentation is the unital associative $\C$-algebra generated by the elements $k, E, F$ satisfying
\eq{ 
[k, E]=2E, \quad [k,F]=-2F, \quad \Big[ E,\big[ E,[F,E]\big]\Big] = 12  E k E,  \quad \Big[ F,\big[ F,[F,E]\big]\Big] = 12 F k F . \label{Y2+:comms}
}
An embedding $\varphi^+:\mcY^+_{2}\hookrightarrow \mcY_{2}$ is provided by:
\eq{
\varphi^+(k) = h, \qq \varphi^+(E) = J(e) - \frac{1}{4}(eh + he), \qq \varphi^+(F) = J(f) + \frac{1}{4} (fh+hf). \label{Y2+:emb}
}
and the left coideal structure on $\mcY^+_{2}$ is given by
$$
\Delta(k)= h\ot1+1\ot k, \quad \Delta(E)= \varphi^+(E)\ot1+1\ot E - \,e\ot k, \quad \Delta(F)= \varphi^+(F)\ot1+1\ot F + \,f\ot k.
$$ 
The counit is given by $\epsilon(E)=\epsilon(F)=0$ and $\epsilon(k)=c$ with $c\in\C$.  $\mcY^+_{2}$ becomes a filtered algebra if we set $\deg(k)=0$ and  $\deg(E)=\deg(F)=1$.

\end{defn}

\begin{defn} 
The symplectic twisted Yangian $\mcY^-_{2}$ in Drinfeld's original presentation is the unital associative $\C$-algebra generated by the elements $h, e, f$ and $G(h), G(e), G(f)$ satisfying
\eqa{ 
& [h,e]=2e, \quad [h,f]=-2f, \quad [e,f]=h, \label{Y2-:Lie} \\
& [h,G(e)]=[G(h),e]=2G(e), \quad [h,G(f)]=[G(h),f]=-2G(f), \quad [e,G(f)]=[G(e),f]=G(h), \label{Y2-:LG} \\
& \big[ G(h),[G(e),G(f)]\big] = 4\big(\{e,G(f),G(h)\}-\{f,G(e),G(h)\}\big). \label{Y2-:GGG} 
}
Here, $\{x_i,x_j,x_k\}$ denotes the normalized totally symmetric polynomial $\frac{1}{6}\mysum_{\pi \in S_3} x_{\pi(i)}x_{\pi(j)}x_{\pi(k)}$. An embedding $\mcY^-_{2}\hookrightarrow \mcY_{2}$ is provided by:
\eq{
\varphi^-(x) = x, \qquad \varphi^-(G(x)) = [J(x'),J(x'')] + \frac{1}{4} \big( [J(x),C] - x\big). \label{Y2-:emb}
}
Here, $C=ef+fe+\frac{1}{2}h^2$ is the quadratic Casimir element. The left coideal structure on $\mcY^-_{2}$ is given by
\eqa{
\Delta(x) &= x\ot1+1\ot x, \el
\Delta(G(x)) &= \varphi^-(G(x))\ot1+1\ot G(x) + [J(x)\ot 1,\Omega] + \frac{1}{4}\Big(\big[ [x\ot1,\Omega],\Omega\big] + \big[ [x'\!\ot1,\Omega],[x''\!\ot1,\Omega]\big]\Big), \nn
}
for $(x,x',x'') = \{ (h,e,f), (e,\frac{h}{2},e),(f,f,\frac{h}{2})  \}$. The counit is given by $\epsilon(x)=\epsilon(G(x))=0$. Yangian $\mcY^-_{2}$ becomes a filtered algebra if we set $\deg(x)=0$ and $\deg(G(x))=2$.
\end{defn}

Here we have chosen a slightly different presentation of $\mcY_{2}^-$ than the one given in \cite{BeCr} because ours gives a more elegant form to the closure relation \eqref{Y2-:GGG}. The isomorphism between the two presentations is given by the map $x\mapsto x$, $G(x) \mapsto K(x) -\frac{1}{4}\,x$.


\subsection{Isomorphism $Y(\mfsp_2)\cong \mcY_2$}

We will use another set of generators of $Y(\mfsp_2)$ to show the isomorphism $Y(\mfsp_2)\cong\mcY_{2}$. The Gaussian decomposition of $T(u)\in \End(\C^2) \ot X(\mfsp_2)[[u^{-1}]]$ is given by 
\eqa{
T(u) = \left(\begin{array}{cc}
1 & 0 \\
\!f(u)\! & 1
\end{array}\right) \!\left(\begin{array}{cc}
\!k_{-1}(u)\! & 0 \\
0 & \!k_1(u)\!
\end{array}\right) \!\left(\begin{array}{cc}
1 & \!e(u)\! \\
0 & 1
\end{array}\right) = \left(\begin{array}{cc}
k_{-1}(u) & k_{-1}(u)\,e(u) \\
\!f(u)\,k_{-1}(u) & \!k_{1}(u)+f(u)\,k_{-1}(u)\,e(u)\!
\end{array}\right) , \label{Ysp2:GD}
}
where the elements $k_{\pm1}(u)$ are invertible. The transposed matrix has the form
\eqa{
T(u)^t = \left(\begin{array}{cc}
k_{1}(u)+f(u)\,k_{-1}(u)\,e(u) & -k_{-1}(u)\,e(u) \\
-f(u)\,k_{-1}(u) & k_{-1}(u)
\end{array}\right) \label{Ysp2:GDt}
}
and a simple calculation gives
\eqa{
T(u)^{-1} = \left(\begin{array}{cc}
k^{-1}_{-1}(u) + e(u)\,k_1^{-1}(u)\,f(u) & -e(u)\,k^{-1}_1(u) \\
-k^{-1}_1(u)\,f(u) & k^{-1}_{1}(u)
\end{array}\right) . \label{Ysp2:inv}
}
Set $k(u)=k_{-1}^{-1}(u) k_1(u)$.

\begin{prop} \label{P:71}
In $X(\mfsp_2)$, we have $[k(u),k(v)] = 0 = [k_1(u),k_1(v)] = [k_{-1}(u),k_{-1}(v)]$ and 
\eqa{
& [k_{-1}(u),f(v)] = \frac{2\left(f(u)-f(v)\right)k_{-1}(u)}{u-v} , \quad [k_{-1}(u),e(v)] = -\frac{2\,k_{-1}(u)\left(e(u)-e(v)\right)}{u-v} , \label{Ysp2:1} \\
& [k_{1}(u),f(v)] = -\frac{2\left(f(u)-f(v)\right)k_{1}(u)}{u-v} , \quad [k_{1}(u),e(v)] = \frac{2\,k_{1}(u)\left(e(u)-e(v)\right)}{u-v} , \label{Ysp2:2} \\
& [k(u) , f(v)] =  - \frac{2}{u-v}\big( (f(u) - f(v))\, k(u) + k(u)\,(f(u) - f(v))\big), \label{Ysp2:3} \\
&  [k(u) , e(v)] =  \frac{2}{u-v}\big( (e(u) - e(v))\, k(u) + k(u)\,(e(u) - e(v))\big), \label{Ysp2:4} \\
& [f(u),f(v)] = -\frac{2\left(f(u)-f(v)\right)^2}{u-v}, \quad [e(u),e(v)] = \frac{2\left(e(u)-e(v)\right)^2}{u-v}, \quad [e(u),f(v)] = \frac{2\left(k(u)-k(v)\right)}{u-v}. \label{Ysp2:5} 
}
\end{prop}

\begin{proof}
The proof is based on calculations similar to those that can be found in \cite[Section 3.1]{Mo4} or~\cite{JiLi}.\qedhere
\end{proof}

Let us expand the series $e(u),f(u)$ and $k(u)$ in the following way:
\eq{ \label{Ysp2:exp}
f(u) = \mysum_{r\geq0} f^{(r)}  u^{-r-1}, \qq
k(u) = 1 + \mysum_{r\geq0} k^{(r)} u^{-r-1}, \qq
e(u) = \mysum_{r\geq0} e^{(r)} u^{-r-1}.
}

Passing to the quotient $X(\mfsp_2) \onto Y(\mfsp_2)$ ($T(u) \mapsto \mcT(u)$) where $\mcT(u+2)^t = \mcT^{-1}(u)$ and keeping the same notation for the new generators leads to $k_{-1}(u+2) = k_1^{-1}(u)$ and $k(u) = k_1(u-2) k_1(u) = k_{-1}^{-1}(u) k_{-1}^{-1}(u+2)$ in $Y(\mfsp_2)$ (compare \eqref{Ysp2:GDt} and \eqref{Ysp2:inv}).

\begin{prop} \label{P:72}
The map $\Phi :  \mcY_{2} \to Y(\mfsp_2) $ given by
\eq{ \label{Ysp2:iso}
\Phi \;:\; \begin{cases}
h \mapsto \frac{1}{2} k^{(0)} , \qq J(h) \mapsto \frac{1}{4} \left(k^{(1)} - \frac{1}{2}(k^{(0)})^2 + \frac{1}{2}(e^{(0)}f^{(0)}+f^{(0)}e^{(0)}) \right)  , \\[.2em]
\hspace{0.5mm} e \mapsto \frac{1}{2} f^{(0)} , \qq J(e) \mapsto  \frac{1}{4} \left(f^{(1)} - \frac{1}{4}(f^{(0)}k^{(0)} + k^{(0)} f^{(0)})\right), \\[.2em]
f \mapsto \frac{1}{2} e^{(0)} , \qq J(f) \mapsto  \frac{1}{4} \left(e^{(1)} - \frac{1}{4}(e^{(0)}k^{(0)} + k^{(0)} e^{(0)})\right)
\end{cases}
}
is an algebra isomorphism.
\end{prop}

\begin{proof}
Relations \eqref{Ysp2:3}-\eqref{Ysp2:5} in Proposition \ref{P:71} are those provided by Drinfeld's second realization of $Y(\mfsl_2)$ \cite{Dr3} via the identification $X^-(u) = e(2u)$, $X^+(v) = f(2v)$ and $H(u) = k(2u)$. The formula for $\Phi$ is the same as the one given in \textit{loc.\ cit.} for the isomorphism between the first and second realizations (that is, between $\mcY_{2}$ and $Y(\mfsl_2)$): it follows that $\Phi$ is a homomorphism. That it is an isomorphism is a consequence of the Poincar\'e-Birkhoff-Witt Theorem for $Y(\mfsp_2)$ \cite{AMR}. 
\end{proof}

We will further make use of the following identities:
\spl{ \label{k-k}
2k^{(0)}_{-1} &= -k^{(0)}, \qq 2k^{(1)}_{-1}=-k^{(1)}- k^{(0)} +\tfrac{3}{4} (k^{(0)})^2, \\
2k^{(2)}_{-1} &= -k^{(2)} -2k^{(1)} + \tfrac{3}{2} k^{(1)} k^{(0)} + \tfrac{3}{2} (k^{(0)})^2 - \tfrac{5}{8}  (k^{(0)})^3  ,
}
which follow by equating coefficients of $u^{-r}$ with $r=1,2,3$ on both sides of $k(u) = k_{-1}^{-1}(u) k_{-1}^{-1}(u+2)$.


\subsection{Isomorphism $Y(\mfsp_2,\mfgl_1)^{tw}\cong\mcY^+_{2}$}

We will show the isomorphism $Y(\mfsp_2,\mfgl_1)^{tw}\cong\mcY^+_{2}$ (here $\mcG = {\rm diag}(1,-1)$) by identifying these as subalgebras of the Yangian of $\mfsl_2$.  Recall that the power series of generators of $Y(\mfsp_2,\mfgl_1)^{tw}$ are denoted by $\si_{ij}(u)$ and they are the entries of the matrix $\Si(u)$. The Gaussian decomposition $\Si(u)=\msF(u)\,\msK(u)\,\msE(u)$ of $\Si(u)$ can be expressed as in  \eqref{Ysp2:GD}. Its transpose and inverse are of the same form as \eqref{Ysp2:GDt} and \eqref{Ysp2:inv}, respectively. The symmetry relation \eqref{RES} 
$$
\Si(u)^t = - \Si(2-u) - \frac{\Si(u)-\Si(2-u)}{2u-2} - \frac{\Tr(\Si(u))\cdot I}{2u-4} 
$$
implies 
\eqa{
\msf(u)\,\msk_{-1}&(u) = \msf(2-u)\,\msk_{-1}(2-u) , \qquad \msk_{-1}(u)\,\mse(u) =  \msk_{-1}(2-u)\,\mse(2-u) \el
& \msk_1(u) + \msf(u)\,\msk_{-1}(u)\,\mse(u) = -\frac{\msk_{-1}(u) + (u-2)\, \msk_{-1}(2-u)}{u-1} . \nn
}
Therefore,
\eq{
\Si(u) = \left(\begin{array}{cc}
\msk_{-1}(u) & \msk_{-1}(u)\,\mse(u) \\
\msf(u)\,\msk_{-1}(u) & \frac{1}{1-u}\big({\msk_{-1}(u)+(2-u)\,\msk_{-1}(2-u)}\big)
\end{array}\right) . \label{Ysp2u1:S(u)}
}
We rewrite the relation $\Si(u)\,\Si(-u)=I$ as $\Si(u)=\Si^{-1}(-u)$, giving
$\msk^{-1}_1(-u)=\msk_1(u)+\msf(u)\, \msk_{-1}(u)\, \mse(u)$. Then, by combining the relations above, we find
$$
\msk^{-1}_1(u) = \frac{\msk_{-1}(-u) - \left(u+2\right) \msk_{-1}(u+2)}{u+1},
$$
hence we need to consider only the power series of generators $\mse(u)$, $\msf(u)$ and $\msk_{-1}(u)$. We set $\msk(u)=\msk_{-1}(u)$ and use the same series expansion as in \eqref{Ysp2:exp}:
\eq{ \label{Ysp2u1:exp}
\msf(u) = \mysum_{r\geq0} \msf^{(r)}  u^{-r-1}, \qquad
\msk(u) = 1+ \mysum_{r\geq0} \msk^{(r)}  u^{-r-1}, \qquad
\mse(u) = \mysum_{r\geq0} \mse^{(r)}  u^{-r-1}.
}

\begin{prop} \label{P:75}
The map $\Phi^+ : \mcY^+_{2}  \to Y(\mfsp_2,\mfgl_1)^{tw}$ given by
\eq{ \label{Ysp2u1:iso} 
\Phi^+ \;:\;  k \mapsto -\tfrac{1}{2} \msk^{(0)} , \quad  F \mapsto -\tfrac{1}{8} \mse^{(1)}  ,  \quad  E \mapsto \tfrac{1}{8} \msf^{(1)}   
}
is an algebra isomorphism.
\end{prop}

\begin{proof}
The embedding $\iota: Y(\mfsp_2,\mfgl_1)^{tw} \into Y(\mfsp_2)$ given by $\Si(u) \mapsto \mcT(u-1)\,\mcG\,\mcT(1-u)^t$, where $\mcG={\rm diag}(1,-1)$, can be made more explicit using \eqref{Ysp2u1:S(u)}:
\eqn{
\msk(u+1) &\mapsto k_{-1}(u)\,e(u)\,f(-u)\,k_{-1}(-u)+k_{-1}(u)\,f(-u)\,k_{-1}(-u)\,e(-u)+k_{-1}(u)\,k_{-1}^{-1}(2-u), \el
\msf(u+1)\, \msk(u+1) &\mapsto  f(u)\,k_{-1}(u)\,e(u)\,f(-u)\,k_{-1}(-u)+f(u)\,k_{-1}(u)\,f(-u)\,k_{-1}(-u)\,e(-u) \el
& \hspace{3.5cm} + f(u)\,k_{-1}(u)\,k_{-1}^{-1}(2-u)+k_{-1}^{-1}(u+2)\,f(-u)\,k_{-1}(-u) , \el
\msk(u+1)\,\mse(u+1) &\mapsto -k_{-1}(u)\,e(u)\,k_{-1}(-u)-k_{-1}(u)\,k_{-1}(-u)\,e(-u).
}
In particular, 
$$
\msk^{(0)} \mapsto 2k_{-1}^{(0)}, \qquad \msf^{(1)} \mapsto  2f^{(1)} + k_{-1}^{(0)} f^{(0)}+3f^{(0)}k_{-1}^{(0)}, \qquad \mse^{(1)} \mapsto -2e^{(1)} - (k^{(0)}_{-1}e^{(0)}-e^{(0)}k^{(0)}_{-1}).
$$
and $\msf^{(0)}\mapsto 0$, $\mse^{(0)}\mapsto 0$, $\msk^{(1)}\mapsto 2 (k^{(0)}_{-1}k^{(0)}_{-1}+2k^{(0)}_{-1})$. By \eqref{k-k} we can substitute $k^{(0)}_{-1}$ with $-\frac{1}{2}k^{(0)}$. Then, using $[k^{(0)},f^{(0)}]=4f^{(0)}$ and $[k^{(0)},e^{(0)}]=-4e^{(0)}$, which follow from \eqref{Ysp2:iso}, we obtain
$$
\msk^{(0)} \mapsto -k^{(0)}, \qquad \msf^{(1)} \mapsto  2f^{(1)} - ( k^{(0)} f^{(0)}+f^{(0)}k^{(0)}) + 2 f^{(0)}, \qquad \mse^{(1)} \mapsto -2e^{(1)} - 2e^{(0)}k^{(0)}.
$$
The composite of the embedding $\iota$ with the isomorphism $\Phi^{-1}: Y(\mfsp_2) \iso \mcY_{2}$ given in Proposition \ref{P:72} sends $\msk^{(0)}, \mse^{(1)}$ and $\msf^{(1)}$ to $-2h, -8J(f)-2(fh+hf) - 4 f$ and $8J(e) - 2(eh+he) + 4 e$ respectively. These elements of $\mcY_{2}$ are the images of $-2k,-8F$ and $8E$ under the composite of the embedding $\varphi^+:\mcY_{2}^+ \into \mcY_{2}$ given in \eqref{Y2+:emb} with the automorphism $\omega$ of $\mcY_{2}$ given by $\omega(x)=x, \, \omega(J(x)) = J(x) + \frac{1}{2} x$ for $x=e,f,h$. It follows that $\mcY^+_{2}$ is isomorphic to $Y(\mfsp_2,\mfgl_1)^{tw}$, that $\Phi^+$ is an algebra isomorphism and that the following diagram is commutative:
$$
\xymatrix{ 
& \mcY^+_{2} \ar[d]^{\Phi^+} \ar[rr]^{\varphi^+} & & \mcY_{2} \ar[d]_{\omega} \\ 
& Y(\mfsp_2,\mfgl_1)^{tw} \ar[r]^{\iota}  &  Y(\mfsp_2) \ar[r]^{\Phi^{-1}} & \mcY_{2} \\
} \qq
$$ 
\end{proof}

\subsection{Isomorphism $Y(\mfsp_2,\mfsp_2)^{tw}\cong\mcY^-_{2}$}

We will show the isomorphism $Y(\mfsp_2,\mfsp_2)^{tw}\cong\mcY^-_{2}$ (here $\mcG$ is the $2\times 2$ identity matrix) in essentially the same way as in the section above. The Gaussian decomposition of $\Si(u)=\msF(u)\,\msK(u)\,\msE(u)\in \End(\C^2) \ot Y(\mfsp_2,\mfsp_2)^{tw}[[u^{-1}]]$ is of the same form as before, only this time $\si^{(0)}_{ij}=\delta_{ij}$. The symmetry relation \eqref{RES} 
$$
\Si(u)^t = \Si(2-u) - \frac{\Si(u)-\Si(2-u)}{2u-2} + \frac{2\,\Si(2-u)-\Tr(\Si(u))\cdot I}{2u-4} 
$$
implies 
\eqn{
\msk_{-1}(u)\,\mse(u) =  \frac{u}{2-u}&\,\msk_{-1}(2-u)\,\mse(2-u) , \qq \msf(u)\,\msk_{-1}(u) = \frac{u}{2-u}\,\msf(2-u)\,\msk_{-1}(2-u) , \\
& \msk_1(u) + \msf(u)\,\msk_{-1}(u)\,\mse(u) = \frac{1}{u-1}  ( \msk_{-1}(u) - u\, \msk_{-1}(2-u)). 
}
We thus obtain
\eq{
\Si(u) = \left(\begin{array}{cc}
\msk_{-1}(u) & \msk_{-1}(u)\,\mse(u) \\
\msf(u)\,\msk_{-1}(u) & \frac{1}{1-u}\big({\msk_{-1}(u)-u\,\msk_{-1}(2-u)}\big)
\end{array}\right) . \label{Ysp2sp2:S(u)}
}
The unitarity relation gives $\msf(u)\, \msk_{-1}(u)\, \mse(u)=\msk^{-1}_1(-u)-\msk_1(u)$. By combining the relations above we find $\msk^{-1}_1(u) = -\tfrac{1}{u+1} (\msk_{-1}(-u)+ u\,\msk_{-1}(u+2))$. As previously, we need to consider only the elements $\mse(u)$, $\msf(u)$ and $\msk_{-1}(u)$. We set $\msk(u)=\msk_{-1}(u)$ and use the series expansion \eqref{Ysp2u1:exp}.

\begin{prop} \label{P:76}
The map $\Phi^- :\mcY^-_{2} \to  Y(\mfsp_2,\mfsp_2)^{tw}$ given by
\eq{ \label{Ysp2sp2:iso} 
\Phi^- \;:\; \begin{cases}
h \mapsto -\frac{1}{2}\msk^{(0)} , \qquad  
G(h) \mapsto  -\frac{1}{8}\msk^{(2)} - \frac{1}{8}(4\msk^{(0)} + 4\msC - \msC \msk^{(0)}) , \\[.25em]
\hspace{0.4mm} 
e \mapsto \frac{1}{4}\msf^{(0)}  , \qquad\quad  
G(e) \mapsto \frac{1}{16}\msf^{(2)} - \frac{1}{16}(4\msf^{(0)} + 4\msk^{(0)}\msf^{(0)} + (\msk^{(0)})^2\msf^{(0)} - \msC \msf^{(0)}), \\[.25em]
f \mapsto \frac{1}{4}\mse^{(0)}, \qquad\hspace{2.8mm}  
G(f)  \mapsto \frac{1}{16}\mse^{(2)} - \frac{1}{16}(4\mse^{(0)} + 4\mse^{(0)}\msk^{(0)} + \mse^{(0)}(\msk^{(0)})^2 - \msC \mse^{(0)}) ,
\end{cases} 
}
where $\msC = \frac{1}{8}\left(\mse^{(0)}\msf^{(0)} + \msf^{(0)}\mse^{(0)} + 2(\msk^{(0)})^2\right)$ is an  algebra isomorphism.
\end{prop}

\begin{proof}
The main ideas of the proof are essentially the same as in the case of Proposition \ref{P:75}. We need to make more explicit the embedding $\iota: Y(\mfsp_2,\mfsp_2)^{tw} \into Y(\mfsp_2)$ given by $\Si(u) \mapsto \mcT(u-1)\, \mcG\, \mcT(1-u)^t$:
\eqn{
\msk(u+1) &\mapsto -k_{-1}(u)\,e(u)\,f(-u)\,k_{-1}(-u)+k_{-1}(u)\,f(-u)\,k_{-1}(-u)\,e(-u)+k_{-1}(u)\,k_{-1}^{-1}(2-u) \\
\msf(u+1)\, \msk(u+1) &\mapsto  -f(u)\,k_{-1}(u)\,e(u)\,f(-u)\,k_{-1}(-u)+f(u)\,k_{-1}(u)\,f(-u)\,k_{-1}(-u)\,e(-u) \\
& \hspace{3.5cm} + f(u)\,k_{-1}(u)\,k_{-1}^{-1}(2-u)-k_{-1}^{-1}(u+2)\,f(-u)\,k_{-1}(-u) , \\
\msk(u+1)\,\mse(u+1) &\mapsto k_{-1}(u)\,e(u)\,k_{-1}(-u)-k_{-1}(u)\,k_{-1}(-u)\,e(-u).
}
In particular, by \eqref{k-k}, we find
\eqn{
\msf^{(0)} \mapsto 2 f^{(0)} , \qquad \msf^{(1)} \mapsto 2(2f^{(0)} + f^{(0)} k^{(0)} ) , \qquad \mse^{(0)} \mapsto 2 e^{(0)} , \qquad \mse^{(1)} \mapsto 2(2e^{(0)} + k^{(0)} e^{(0)}), \\
\msk^{(0)} \mapsto - k^{(0)}, \qquad \msk^{(1)} \mapsto  e^{(0)} f^{(0)} +  f^{(0)} e^{(0)}  - 2k^{(0)} + \tfrac{1}{2}(k^{(0)})^2 , \hspace{2.7cm}
}
and
$$
\msk^{(2)} \mapsto - k^{(2)} + [e^{(1)}, f^{(0)}] - e^{(0)} f^{(1)} -  f^{(1)}e^{(0)} + k^{(0)}k^{(1)}  +   2(k^{(0)})^2 - \tfrac{1}{2}(k^{(0)})^3 + 2e^{(0)} f^{(0)} - 3k^{(0)}.
$$
Consider the following diagram where $\omega:\mcY_{2} \lra \mcY_{2}$ is given by $\omega(x) =x, \, \omega(J(x)) = J(x) + \frac{1}{2}x, \, x\in \mcY_{2}$: 
$$
\xymatrix{ 
& \mcY^-_{2} \ar[d]^{\Phi^-} \ar[rr]^{\varphi^-} & & \mcY_{2} \ar[d]_{\omega} \\ 
& Y(\mfsp_2,\mfsp_2)^{tw} \ar[r]^{\iota}  &  Y(\mfsp_2) \ar[r]^{\Phi^{-1}} & \mcY_{2} \\
} \qq
$$
\smallskip

\noindent It follows from the defining relation of $Y(\mfsp_2,\mfsp_2)^{tw}$ ((4.4) in \cite{GR}) combined with \eqref{Ysp2sp2:S(u)} that $Y(\mfsp_2,\mfsp_2)^{tw}$ is generated by $\msk^0, \mse^{(0)}, \msf^{(0)}$ and $\msk^{(2)}$. Therefore, the image of the embedding $ \Phi^{-1} \circ \iota$ is the subalgebra of $ \mcY_{2}$ generated by $e,f,h$ and $\Phi^{-1}(\iota(\msk^{(2)}))$. 

The Yangian $\mcY^-_{2}$ is generated by $h,e,f$ and $G(h)$: see \eqref{Y2-:LG}. From the definition of $\omega$, the embedding $\varphi^-$ given in \eqref{Y2-:emb} and after computing that 
\eqn{ 
\Phi^{-1}(\iota(\msk^{(2)})) & = - 8[J(e),J(f)] - 8(J(h) + \tfrac{1}{4}[J(h),C]) - 4 (2h - 2C + Ch) \\
& = -8\omega(\varphi^-(G(h))) - 4( 2h - 2C + Ch) ,
}
(where $C$ is the Casimir element), we can see that the image of the embedding $\omega \circ \varphi^-$ is the same as the image of $ \Phi^{-1} \circ \iota$.  Therefore, $ \mcY^-_{2}$ and $Y(\mfsp_2,\mfsp_2)^{tw}$ are isomorphic and $\Phi^-$ defines an isomorphism of algebras which makes the diagram above commutative.

The images of the elements $\mse^{(2)}$ and $\msf^{(2)}$ are obtained by the repeating the same steps as above. In particular, we find
\eqn{
\mse^{(2)} &\mapsto 2 e^{(2)} - \tfrac{1}{2} \big([k^{(0)}, e^{(1)}] +  e^{(0)} k^{(1)} +3 k^{(1)} e^{(0)}\big) - 2 \big(e^{(0)} f^{(0)} +f^{(0)} e^{(0)} -4 k^{(0)} - (k^{(0)})^2 -5 \big)e^{(0)} , \\
\msf^{(2)} &\mapsto 2 f^{(2)} -\tfrac{1}{2} \big([f^{(0)}, k^{(1)}] + k^{(0)} f^{(1)}+3 f^{(1)} k^{(0)}\big) - f^{(0)}  \big(e^{(0)} f^{(0)} + f^{(0)} e^{(0)} - 6 k^{(0)} - 2 (k^{(0)})^2 - 2\big) ,
}
leading to
\eqn{ 
\Phi^{-1}(\iota(\mse^{(2)})) & = 16\omega(\varphi^-(G(f))) + 8 ( 2f - 4 fh + 2fh^2 - Cf) , \\
\Phi^{-1}(\iota(\msf^{(2)})) & = 16\omega(\varphi^-(G(e))) + 8 ( 2e - 4 he + 2h^2e - Ce) .
}
\end{proof}

We want to end this section by emphasizing an open question regarding the commutation relations of the generators $\msf(u)$, $\mse(u)$, $\msk(u)$. For the corresponding elements of the (non-twisted) Yangian the commutation relations are given by Proposition \ref{P:71}. It would be very interesting to find an equivalent set of relations for the twisted Yangian. An elegant set of such relations would be a good starting point in constructing an analogue of Drinfeld's new presentation for twisted Yangians \cite{Dr3}.

\newpage



\end{document}